\documentclass[12pt]{amsart}
\usepackage{amsmath}
\usepackage{amsxtra}
\usepackage{amscd}
\usepackage{amsthm}
\usepackage{amsfonts}
\usepackage{amssymb}%,mathabx}
\usepackage{eucal}
\usepackage{mathrsfs}
\usepackage{young}
\usepackage[boxsize=.2in]{ytableau}
\usepackage{verbatim}
\usepackage[all]{xy}
\usepackage{color}
\definecolor{deepjunglegreen}{rgb}{0.0, 0.29, 0.29}
\definecolor{darkspringgreen}{rgb}{0.09, 0.45, 0.27}

\usepackage{stmaryrd}

\makeatletter
\usepackage{etex,etoolbox,needspace}
\pretocmd\section{\Needspace*{4\baselineskip}}{}{}
\makeatletter

\usepackage[pdftex,bookmarks=false,colorlinks=true,citecolor=darkspringgreen,debug=true,
  naturalnames=true,pdfnewwindow=true]{hyperref}

\newtheorem{thm}{Theorem}[subsection]

\newtheorem{cor}[thm]{Corollary}

\newtheorem{lem}[thm]{Lemma}
\newtheorem{prop}[thm]{Proposition}

\theoremstyle{definition}
\newtheorem{defn}[thm]{Definition}

\theoremstyle{remark}
\newtheorem{rem}[thm]{Remark}
\newcommand{\nc}{\newcommand}
\nc{\renc}{\renewcommand} \nc{\ssec}{\subsection}
\nc{\sssec}{\subsubsection}

\nc{\on}{\operatorname} \nc{\wh}{\widehat}
\nc\ol{\overline} \nc\ul{\underline} \nc\wt{\widetilde}

\newcommand{\red}[1]{{\color{red}#1}}

\nc{\BA}{{\mathbb{A}}} \nc{\BC}{{\mathbb{C}}} \nc{\BF}{{\mathbb{F}}}
\nc{\BD}{{\mathbb{D}}} \nc{\BG}{{\mathbb{G}}} \nc{\BQ}{{\mathbb{Q}}}
\nc{\BM}{{\mathbb{M}}} \nc{\BN}{{\mathbb{N}}} \nc{\BO}{{\mathbb{O}}}
\nc{\BP}{{\mathbb{P}}} \nc{\BR}{{\mathbb{R}}}
\nc{\BZ}{{\mathbb{Z}}} \nc{\BS}{{\mathbb{S}}} \nc{\BW}{{\mathbb{W}}}

\nc{\CA}{{\mathcal{A}}} \nc{\CB}{{\mathcal{B}}} \nc{\CalC}{{\mathcal{C}}} \nc{\CalD}{{\mathcal{D}}}
\nc{\CE}{{\mathcal{E}}} \nc{\CF}{{\mathcal{F}}} \nc{\CC}{{\mathcal{C}}} 
\nc{\CG}{{\mathcal{G}}} \nc{\CH}{{\mathcal{H}}}
\nc{\CI}{{\mathcal{I}}} \nc{\CK}{{\mathcal{K}}} \nc{\CL}{{\mathcal{L}}}
\nc{\CM}{{\mathcal{M}}} \nc{\CN}{{\mathcal{N}}}
\nc{\CO}{{\mathcal{O}}} \nc{\CP}{{\mathcal{P}}}
\nc{\CQ}{{\mathcal{Q}}} \nc{\CR}{{\mathcal{R}}}
\nc{\CS}{{\mathcal{S}}} \nc{\CT}{{\mathcal{T}}}
\nc{\CU}{{\mathcal{U}}} \nc{\CV}{{\mathcal{V}}}  \nc{\CY}{{\mathcal Y}}
\nc{\CW}{{\mathcal{W}}} \nc{\CZ}{{\mathcal{Z}}}

\nc{\cM}{{\check{\mathcal M}}{}} \nc{\csM}{{\check{\mathcal A}}{}}
\nc{\oM}{{\overset{\circ}{\mathcal M}}{}}
\nc{\obM}{{\overset{\circ}{\mathbf M}}{}}
\nc{\oCA}{{\overset{\circ}{\mathcal A}}{}}
\nc{\obA}{{\overset{\circ}{\mathbf A}}{}}
\nc{\ooM}{{\overset{\circ}{M}}{}}
\nc{\osM}{{\overset{\circ}{\mathsf M}}{}}
\nc{\vM}{{\overset{\bullet}{\mathcal M}}{}}
\nc{\nM}{{\underset{\bullet}{\mathcal M}}{}}
\nc{\oD}{{\overset{\circ}{\mathcal D}}{}}
\nc{\obD}{{\overset{\circ}{\mathbf D}}{}}
\nc{\oA}{{\overset{\circ}{\mathbb A}}{}}
\nc{\op}{{\overset{\bullet}{\mathbf p}}{}}
\nc{\cp}{{\overset{\circ}{\mathbf p}}{}}
\nc{\oU}{{\overset{\bullet}{\mathcal U}}{}}
\nc{\ofZ}{{\overset{\circ}{\mathfrak Z}}{}}

\nc{\ff}{{\mathfrak{f}}} \nc{\fv}{{\mathfrak{v}}}
\nc{\fa}{{\mathfrak{a}}} \nc{\fb}{{\mathfrak{b}}}
\nc{\fd}{{\mathfrak{d}}} \nc{\fe}{{\mathfrak{e}}}
\nc{\fg}{{\mathfrak{g}}} \nc{\fgl}{{\mathfrak{gl}}}
\nc{\fh}{{\mathfrak{h}}} \nc{\fri}{{\mathfrak{i}}}
\nc{\fj}{{\mathfrak{j}}} \nc{\fk}{{\mathfrak{k}}} \nc{\fl}{{\mathfrak{l}}}
\nc{\fm}{{\mathfrak{m}}} \nc{\fn}{{\mathfrak{n}}}
\nc{\ft}{{\mathfrak{t}}} \nc{\fu}{{\mathfrak{u}}}
\nc{\fw}{{\mathfrak{w}}} \nc{\fz}{{\mathfrak{z}}}
\nc{\fp}{{\mathfrak{p}}} \nc{\fq}{{\mathfrak{q}}} \nc{\frr}{{\mathfrak{r}}}
\nc{\fs}{{\mathfrak{s}}} \nc{\fsl}{{\mathfrak{sl}}}
\nc{\fso}{{\mathfrak{so}}} \nc{\fsp}{{\mathfrak{sp}}} \nc{\osp}{{\mathfrak{osp}}}
\nc{\hsl}{{\widehat{\mathfrak{sl}}}}
\nc{\hgl}{{\widehat{\mathfrak{gl}}}}
\nc{\hg}{{\widehat{\mathfrak{g}}}}
\nc{\chg}{{\widehat{\mathfrak{g}}}{}^\vee}
\nc{\hn}{{\widehat{\mathfrak{n}}}}
\nc{\chn}{{\widehat{\mathfrak{n}}}{}^\vee}

\nc{\fA}{{\mathfrak{A}}} \nc{\fB}{{\mathfrak{B}}} \nc{\fC}{{\mathfrak{C}}}
\nc{\fD}{{\mathfrak{D}}} \nc{\fE}{{\mathfrak{E}}}
\nc{\fF}{{\mathfrak{F}}} \nc{\fG}{{\mathfrak{G}}} \nc{\fH}{{\mathfrak{H}}}
\nc{\fI}{{\mathfrak{I}}} \nc{\fJ}{{\mathfrak{J}}}
\nc{\fK}{{\mathfrak{K}}} \nc{\fL}{{\mathfrak{L}}}
\nc{\fM}{{\mathfrak{M}}} \nc{\fN}{{\mathfrak{N}}}
\nc{\fP}{{\mathfrak{P}}} \nc{\fQ}{{\mathfrak{Q}}}
\nc{\fS}{{\mathfrak{S}}} \nc{\fT}{{\mathfrak{T}}} \nc{\fU}{{\mathfrak{U}}}
\nc{\fV}{{\mathfrak{V}}} \nc{\fW}{{\mathfrak{W}}}
\nc{\fX}{{\mathfrak{X}}} \nc{\fY}{{\mathfrak{Y}}}
\nc{\fZ}{{\mathfrak{Z}}}

\nc{\ba}{{\mathbf{a}}}
\nc{\bb}{{\mathbf{b}}} \nc{\bc}{{\mathbf{c}}} \nc{\be}{{\mathbf{e}}}
\nc{\bg}{{\mathbf{g}}} \nc{\bj}{{\mathbf{j}}} \nc{\bm}{{\mathbf{m}}}
\nc{\bn}{{\mathbf{n}}} \nc{\bp}{{\mathbf{p}}}
\nc{\bq}{{\mathbf{q}}} \nc{\br}{{\mathbf{r}}} \nc{\bt}{{\mathbf{t}}}
\nc{\bfu}{{\mathbf{u}}} \nc{\bv}{{\mathbf{v}}}
\nc{\bx}{{\mathbf{x}}} \nc{\by}{{\mathbf{y}}} \nc{\bz}{{\mathbf{z}}}
\nc{\bw}{{\mathbf{w}}} \nc{\bA}{{\mathbf{A}}}
\nc{\bB}{{\mathbf{B}}} \nc{\bC}{{\mathbf{C}}}
\nc{\bD}{{\mathbf{D}}} \nc{\bF}{{\mathbf{F}}} \nc{\bG}{{\mathbf{G}}}
\nc{\bH}{{\mathbf{H}}} \nc{\bI}{{\mathbf{I}}} \nc{\bJ}{{\mathbf{J}}}
\nc{\bK}{{\mathbf{K}}} \nc{\bM}{{\mathbf{M}}} \nc{\bN}{{\mathbf{N}}}
\nc{\bO}{{\mathbf{O}}} \nc{\bS}{{\mathbf{S}}} \nc{\bT}{{\mathbf{T}}}
\nc{\bU}{{\mathbf{U}}} \nc{\bV}{{\mathbf{V}}} \nc{\bW}{{\mathbf{W}}}
\nc{\bX}{{\mathbf{X}}}
\nc{\bY}{{\mathbf{Y}}} \nc{\bP}{{\mathbf{P}}}
\nc{\bZ}{{\mathbf{Z}}} \nc{\bh}{{\mathbf{h}}}

\nc{\sA}{{\mathsf{A}}} \nc{\sB}{{\mathsf{B}}}
\nc{\sC}{{\mathsf{C}}} \nc{\sD}{{\mathsf{D}}}
\nc{\sE}{{\mathsf{E}}} \nc{\sF}{{\mathsf{F}}} \nc{\sG}{{\mathsf{G}}}
\nc{\sI}{{\mathsf{I}}} \nc{\sk}{{\mathsf{k}}} \nc{\sK}{{\mathsf{K}}} \nc{\sL}{{\mathsf{L}}}
\nc{\sfm}{{\mathsf{m}}} \nc{\sM}{{\mathsf{M}}} \nc{\sN}{{\mathsf{N}}}
\nc{\sO}{{\mathsf{O}}} \nc{\sQ}{{\mathsf{Q}}} \nc{\sP}{{\mathsf{P}}}
\nc{\sT}{{\mathsf{T}}} \nc{\sZ}{{\mathsf{Z}}}
\nc{\sV}{{\mathsf{V}}} \nc{\sW}{{\mathsf{W}}}
\nc{\sfp}{{\mathsf{p}}} \nc{\sq}{{\mathsf{q}}} \nc{\sr}{{\mathsf{r}}}
\nc{\sfs}{{\mathsf{s}}} \nc{\st}{{\mathsf{t}}} \nc{\sfb}{{\mathsf{b}}}
\nc{\sfc}{{\mathsf{c}}} \nc{\sd}{{\mathsf{d}}}
\nc{\sz}{{\mathsf{z}}}

\nc{\tA}{{\widetilde{\mathbf{A}}}}
\nc{\tB}{{\widetilde{\mathcal{B}}}}
\nc{\tg}{{\widetilde{\mathfrak{g}}}} \nc{\tG}{{\widetilde{G}}}
\nc{\TM}{{\widetilde{\mathbb{M}}}{}}
\nc{\tO}{{\widetilde{\mathsf{O}}}{}}
\nc{\tU}{{\widetilde{\mathfrak{U}}}{}} \nc{\TZ}{{\tilde{Z}}}
\nc{\tx}{{\tilde{x}}} \nc{\tbv}{{\tilde{\bv}}}
\nc{\tfP}{{\widetilde{\mathfrak{P}}}{}} \nc{\tz}{{\tilde{\zeta}}}
\nc{\tmu}{{\tilde{\mu}}}

\nc{\urho}{\underline{\rho}} \nc{\uB}{\underline{B}}
\nc{\uC}{{\underline{\mathbb{C}}}} \nc{\ui}{\underline{i}}
\nc{\uj}{\underline{j}} \nc{\ofP}{{\overline{\mathfrak{P}}}}
\nc{\oB}{{\overline{\mathcal{B}}}}
\nc{\og}{{\overline{\mathfrak{g}}}} \nc{\oI}{{\overline{I}}}

\nc{\eps}{\varepsilon} \nc{\hrho}{{\hat{\rho}}}
\nc{\blambda}{{\boldsymbol{\lambda}}} \nc{\bmu}{{\boldsymbol{\mu}}} \nc{\bnu}{{\boldsymbol{\nu}}}
\nc{\btheta}{{\boldsymbol{\theta}}} \nc{\bzeta}{{\boldsymbol{\zeta}}} \nc{\bta}{{\boldsymbol{\eta}}}

\nc{\one}{{\mathbf{1}}} \nc{\two}{{\mathbf{t}}}

\nc{\Sym}{\mathop{\operatorname{\rm Sym}}}
\nc{\Tot}{{\mathop{\operatorname{\rm Tot}}}}
\nc{\Spec}{\mathop{\operatorname{\rm Spec}}}
\nc{\Ker}{{\mathop{\operatorname{\rm Ker}}}}
\nc{\Isom}{{\mathop{\operatorname{\rm Isom}}}}
\nc{\Hilb}{{\mathop{\operatorname{\rm Hilb}}}}
\nc{\deeq}{{\mathop{\operatorname{\rm deeq}}}}
\nc{\End}{{\mathop{\operatorname{\rm End}}}}
\nc{\Ext}{{\mathop{\operatorname{\rm Ext}}}}
\nc{\Hom}{{\mathop{\operatorname{\rm Hom}}}}
\nc{\CHom}{{\mathop{\operatorname{{\mathcal{H}}\it om}}}}
\nc{\GL}{{\mathop{\operatorname{\rm GL}}}}
\nc{\SL}{{\mathop{\operatorname{\rm SL}}}}
\nc{\SO}{{\mathop{\operatorname{\rm SO}}}}
\nc{\Sp}{{\mathop{\operatorname{\rm Sp}}}}
\nc{\OSp}{{\mathop{\operatorname{\rm SOSp}}}}
\nc{\gr}{{\mathop{\operatorname{\rm gr}}}}
\nc{\Id}{{\mathop{\operatorname{\rm Id}}}}
\nc{\perf}{{\mathop{\operatorname{\rm perf}}}}
\nc{\defi}{{\mathop{\operatorname{\rm def}}}}
\nc{\length}{{\mathop{\operatorname{\rm length}}}}
\nc{\supp}{{\mathop{\operatorname{\rm supp}}}}
\nc{\HC}{{\mathcal H}{\mathcal C}}

\nc{\pr}{{\operatorname{pr}}}
\nc{\Cliff}{{\mathsf{Cliff}}}
\nc{\loc}{{\operatorname{loc}}}
\nc{\Fl}{{\mathbf{Fl}}} \nc{\Ffl}{{\mathcal{F}\ell}}
\nc{\Fib}{{\mathsf{Fib}}}
\nc{\Coh}{{\mathsf{Coh}}} \nc{\FCoh}{{\mathsf{FCoh}}}
\nc{\Perf}{{\mathsf{Perf}}}
\nc{\wtimes}{\mathbin{\widetilde\times}}

\nc{\reg}{{\text{\rm reg}}}
\nc{\self}{{\text{\rm self}}}
\nc{\gvee}{{\mathfrak g}^{\!\scriptscriptstyle\vee}}
\nc{\tvee}{{\mathfrak t}^{\!\scriptscriptstyle\vee}}
\nc{\nvee}{{\mathfrak n}^{\!\scriptscriptstyle\vee}}
\nc{\bvee}{{\mathfrak b}^{\!\scriptscriptstyle\vee}}
       \nc{\rhovee}{\rho^{\!\scriptscriptstyle\vee}}

\nc{\cplus}{{\mathbf{C}_+}} \nc{\cminus}{{\mathbf{C}_-}}
\nc{\cthree}{{\mathbf{C}_*}} \nc{\Qbar}{{\bar{Q}}}

\newcommand\iso{\mathbin{\vphantom{j^{X^2}}\smash{\overset{\sim}{\vphantom{\rule{0pt}{0.20em}}\smash{\longrightarrow}}}}}
\nc{\Gtimes}{\vphantom{j^{X^2}}\smash{\overset{G}{\vphantom{\rule{0pt}{0.30em}}\smash{\times}}}}
\nc{\sGtimes}{\vphantom{j^{X^2}}\smash{\overset{\mathsf G}{\vphantom{\rule{0pt}{0.30em}}\smash{\times}}}}

\nc{\bOmega}{{\overline{\Omega}}}

\nc{\seq}[1]{\stackrel{#1}{\sim}}

\nc{\aff}{{\operatorname{aff}}} \nc{\ass}{{\operatorname{ass}}} \nc{\Hyp}{{\operatorname{Hyp}}}
\nc{\fin}{{\operatorname{fin}}} \nc{\GGM}{{\operatorname{GGM}}}
\nc{\mir}{{\operatorname{mir}}}
\nc{\triv}{{\operatorname{triv}}}
\nc{\ext}{{\operatorname{ext}}}
\nc{\righ}{{\operatorname{right}}}
\nc{\lef}{{\operatorname{left}}}
\nc{\forg}{{\operatorname{forg}}}
\nc{\fid}{{\operatorname{fd}}}
\nc{\odd}{{\operatorname{odd}}}
\nc{\even}{{\operatorname{even}}}
\nc{\modu}{{\operatorname{-mod}}}
\nc{\Gr}{{\mathbf{Gr}}}
\nc{\FT}{{\operatorname{FT}}}
\nc{\Mat}{{\operatorname{Mat}}}
\nc{\MSt}{{\operatorname{MSt}}}
\nc{\sph}{{\operatorname{sph}}}
\nc{\GR}{{\mathbf{Gr}}}
\nc{\Perv}{{\operatorname{Perv}}}
\nc{\Rep}{{\operatorname{Rep}}}
\nc{\Ind}{{\operatorname{Ind}}}
\nc{\IC}{{\operatorname{IC}}}
\nc{\Bun}{{\operatorname{Bun}}}
\nc{\Proj}{{\operatorname{Proj}}}
\nc{\Stab}{{\operatorname{Stab}}}
\nc{\Vect}{{\operatorname{Vect}}}
\nc{\pt}{{\operatorname{pt}}}
\nc{\bfmu}{{\boldsymbol{\mu}}}
\nc{\bfomega}{{\boldsymbol{\omega}}}
\nc{\calM}{\mathcal M}
\nc{\calA}{\mathcal A}
\nc{\calO}{\mathcal O}
\nc{\calN}{\mathcal N}
\nc{\grg}{\mathfrak g}
\nc{\dslash}{/\!\!/}
\nc{\tslash}{/\!\!/\!\!/}
\nc\grt{\mathfrak t}
\nc\bfM{\mathbf M}
\nc\bfN{\mathbf N}
\nc\Sig{\Sigma}
\nc\ZZ{\mathbb{Z}}
\nc\calC{\mathcal C}
\nc\calF{\mathcal F}
\nc\calX{\mathcal X}
\nc\calY{\mathcal Y}
\nc\QCoh{\operatorname{QCoh}}
\nc\IndCoh{\operatorname{IndCoh}}
\nc\Maps{\operatorname{Maps}}
\nc\Dmod{D-\operatorname{mod}}
\newcommand\Hecke{\operatorname{Hecke}}
\nc{\calD}{\mathcal D}
\nc\bfO{\mathbf O}

\nc\GG{\mathbb G}
\nc\calK{\mathcal K}
\nc{\calG}{\mathcal G}
\nc\RHom{\operatorname{RHom}}
\nc\Res{\operatorname{Res}}
\nc\Av{\operatorname{Av}}
\nc\grs{\mathfrak s}
\nc{\tilX}{\widetilde X}
\nc\calB{\mathcal B}
\nc\calS{\mathcal S}
\nc\calT{\mathcal T}
\nc\calZ{\mathcal Z}
\nc\LS{\operatorname{LocSys}}
\nc\bfL{\on{\mathbf L}}

\nc{\scP}{{\mathscr{P}}}
\nc{\scQ}{{\mathscr{Q}}}

\newcommand*\circled[1]%{\tikz[baseline=(char.base)]{
{*{#1}}
\theoremstyle{remark}
\newtheorem{example}[thm]{Example}

\def\Des{\mathrm{Des}}

\def\Ind{\mathrm{Ind}}
\def\Sym{\mathrm{Sym}}
\def\Asym{\mathrm{Asym}}
\def\Image{\mathrm{Im}}
\def\Kernel{\mathrm{Ker}}
\def\Span#1{\underset{#1}{\mathrm{Span}}\,}

\makeatletter
\newcommand{\raisemath}[1]{\mathpalette{\raisem@th{#1}}}
\newcommand{\raisem@th}[3]{\raisebox{#1}{$#2#3$}}
\nc{\binlim}[2][]{\def\@tempa{#1}\@ifnextchar^{\@binlim{#2}}{\@binlim{#2}^{}}}
\def\@binlim#1^#2{\mathbin{\@ifempty{#2}{\mathop{#1}}{\mathop{#1}\@xp\displaylimits\@tempa^{#2}}}}

\makeatother

\nc\cX{{\mathcal X}}

\nc\Gm{{\mathbb G_m}}
\nc{\isoto}%{\mathbin{\widetilde\to}} %
    {\stackrel{\displaystyle\raisemath{-.2ex}{\sim}}\to}
\nc\cD\CalD
\nc\setm\setminus
\nc\cE\CE
\renc\Hecke{\mathit{\CH\kern-.2ex ecke}}
\nc\bsl\backslash
\nc\Fq{\mathbb F_q}
\nc\x\times
\nc\bGO{{\bG_\bO}}
\nc\bla\blambda
\nc\blt\bullet
\nc\opp{{\on{op}}}
\nc\srel\stackrel
\nc\tbx{\binlim{\widetilde\boxtimes{}}}
\nc\into\hookrightarrow
\nc\otto\leftrightarrow
\renc\setminus\smallsetminus
\nc\phitau{\varphi\tau}
\newenvironment{i-ii-iii}{%
\begin{enumerate}
}%
{\end{enumerate}}
\nc\ceil[1]{\lceil#1\rceil}  \nc\floor[1]{\lfloor#1\rfloor}
\nc\Lie{\on{Lie}}
\nc\sS{{\mathsf S}}
\nc\vvv{\ensuremath{\red\surd}}
\nc\la\lambda
 
\nc\kap{\kappa}
\nc\gra{\mathfrak a}
\nc\gl{\mathfrak{gl}}
\nc\sTr{\operatorname{sTr}}
\nc\hatG{\widehat{G}}
\nc\calL{\mathcal L}
\nc\Whit{\operatorname{Whit}}
\nc\KL{\operatorname{KL}}

\makeatletter
\renewcommand{\subsection}{\@startsection{subsection}{2}{0pt}{-3ex
plus -1ex minus -0.2ex}{-2mm plus -0pt minus
-2pt}{\normalfont\bfseries}} \makeatother

\numberwithin{equation}{subsection}
\allowdisplaybreaks
\nc\mto{\mapsto }
\nc\en{\enspace }

\begin{document}

\author[M.~Finkelberg]{Michael Finkelberg}
\address{Einstein Institute of Mathematics, The Hebrew University of Jerusalem,
  Edmond J. Safra Campus, Giv’at Ram, Jerusalem, 91904, Israel;
\newline  National Research University Higher School of Economics}
\email{fnklberg@gmail.com}

\author[A.~Postnikov]{Alexander Postnikov}
\address{Department of Mathematics, Massachusetts Institute of Technology, 77 Massachusetts Ave, Cambridge MA 02139, USA}
\email{apost@math.mit.edu}

\author[V.~Schechtman]{Vadim Schechtman}
\address{Institut de Math\'ematiques de Toulouse, Universit\'e Paul Sabatier, 118 Route de Narbonne, 31062 Toulouse, France;
\newline Kavli IPMU (WPI), UTIAS, University of Tokyo, 5-1-5 Kashiwanoha,
Kashiwa, Chiba, 277-8583 Japan}
\email{schechtman@math.ups-tlse.fr}

\title
{Kostka numbers and Fourier duality}
\dedicatory{To the blessed memory of Yuri Ivanovich Manin}

%\thanks{{\bf Mathematics Subject Classification (2000).}
%19E08, (22E65, 37K10).}

%\thanks{{\bf Key words.} $q$-difference Toda lattice, Equivariant
%$K$-theory, Laumon compactification.}

%\thanks{The work of L.R. was partially supported
%by  RFBR grants 07-01-92214-CNRSL-a and 05-01-02805-CNRSL-a.
%L.R. gratefully acknowledges the support from Deligne 2004 Balzan
%prize in mathematics.}

\begin{abstract}
We relate the Fourier transform of perverse sheaves smooth along the coordinate
hyperplane configuration in a complex vector space to the Deligne-Lusztig duality
of unipotent representations of a general linear group over a finite field.
A similar relation is established for arbitrary finite Coxeter groups.  
%We interpret the M\"obius inversion formula as a Morita equivalence between
%two linear algebra avatars of some perverse sheaves over the coordinate arrangements
%in $\BC^n$. As an illustration we present a reformulation of a classical duality on the
%category of representations of Chevalley groups as a Fourier transform.
\end{abstract}

\maketitle

\tableofcontents

\section{Introduction}
\label{intro}

\subsection{}
As soon as perverse sheaves were discovered about 40 years ago, there appeared a problem of concrete
linear algebraic description of abelian categories of perverse sheaves smooth along some particular
stratifications (``How to glue perverse sheaves''). One important case is a complex affine space
stratified by a hyperplane arrangement. In case such an arrangement is {\em real}, a solution of
this problem was suggested in~\cite{ks1} ({\em hyperbolic sheaves}).

In the present note we study a most trivial example of a hyperplane arrangement: namely the one of
coordinate hyperplanes. In this case another description of the category of perverse sheaves was
proposed in~\cite{ggm} ({\em iterated vanishing/nearby cycles}). It turns out that already in this
simplest example, the interplay between the two different descriptions of the same abelian category
leads to some nontrivial combinatorial consequences.

Namely, we go even further along the route of simplification and consider semisimple perverse sheaves
on $\BC^n$ with trivial monodromy, however with coefficients not in the category of vector spaces,
but in the category $\CC$ of representations of the symmetric group $S_{n+1}$ (over a field $\sk$ of
characteristic~0). We consider a perverse sheaf $\CF_\mu$ whose generic stalk is the irreducible
$S_{n+1}$-module $V_\mu$ (for a partition $\mu$ of $n+1$), and whose hyperbolic stalks are 
equal to induction of $V_\mu$ to $S_{n+1}$ from appropriate parabolic subgroups. When
$\mu=(n+1)$, so that $V_{(n+1)}$ is trivial, we call the multiplicities of simple perverse
sheaves in $\CF_{(n+1)}$ {\em small Kostka numbers} $\kappa_{\lambda,I}$.
According to~Theorem~\ref{1.6}, $\kappa_{\lambda,I}$ is equal to the number of standard Young tableaux
of shape $\lambda$ with descent set $I$. 
We also discuss an extension of (small) 
Kostka numbers to an arbitrary finite Coxeter group $W$
related to \cite{so2}.

The {\em hyperbolic calculus} developed in~\cite{fks} allows to compute the Fourier-Sato transform
of $\CF_\mu$. Namely, we have $\on{FT}\CF_\mu\cong\CF_{\mu^t}$ (transposed partition). Also, the
hyperbolic calculus provides a {\em master complex} $\CM as^\bullet_\mu$ computing the vanishing
cycles $\Phi_0\CF_\mu$ of $\CF_\mu$ at the origin. It turns out that $\Phi_0\CF_\mu\cong V_{\mu^t}$,
and the master complex goes back at least to~\cite{so}.

For a prime power $q$, there is a Benson-Curtis isomorphism
$\sk[S_{n+1}]\iso\CH_{n+1}$ with the Iwahori-Hecke algebra. Thus $\CC$ is equivalent to the category
$\CC_q$ of $\CH_{n+1}$-modules. Under this equivalence, the master complex $\CM as^\bullet_\mu$
corresponds to a complex $_q\CM as^\bullet_\mu$ going back at least to~\cite{ka}. According to
{\em loc.cit.}, a similar complex
exists for the Hecke algebra of an arbitrary finite Coxeter group.

Furthermore, $\CC_q$ is equivalent to the category of unipotent representations of
$\GL(n+1,\BF_q)$, and under this equivalence, $_q\CM as^\bullet_\mu$ corresponds to the complex
of~\cite{dl} defining the Alvis-Curtis-Deligne-Lusztig-Kawanaka duality~\cite{a,c,kaw,dl}.

Finally, note that in the setup of representations of $p$-adic reductive groups, the
Deligne-Lusztig duality coincides with the Serre functor (up to a shift)~\cite{bbk}. On the other hand,
in the setup of perverse sheaves smooth along the coordinate hyperplanes stratification of
$\BC^n$, the Serre functor coincides with the Fourier-Sato transform
(up to a shift)~\cite{k}. The relation of Deligne-Lusztig duality with Fourier-Sato transform we
discuss in the present note may be viewed as a baby mock version of the comparison of
these deep results. 

\subsection{Acknowledgments}
This note was conceived in 2017. We are grateful to 
Anatol Kirillov and Mikhail Kapranov for the inspiring discussions.
We also thank George Lusztig for bringing~\cite{so2} to our attention.

The research of M.F.\ was supported by the Israel Science Foundation (grant No.~994/24) and BSF (grant No.~2024252).

\section{Hyperbolic sheaves over coordinate arrangements}
\label{coordinate}

If $(X, \leq)$ is a poset, we call a {\it bisheaf} on $X$ with values in a category $\CC$ 
a couple $B = (\gamma, \delta)$ of functors
\[\gamma\colon  X \to \CC,\ \delta\colon X^\opp \to \CC\]
such that for all $x\in X,\ \gamma(x) = \delta(x)=:B(x)$. Thus, for all 
$x\leq y$ we have two arrows
\[\gamma_{x,y}\colon B(x)\leftrightarrows B(y)\ :\delta_{y,x}.\]

\subsection{Real coordinate arrangement}
\label{2.1}
Let $V = \BR^n$ with coordinates $x_1, \ldots, x_n$;
\[\CH = \{H_1, \ldots, H_n\},\]
where $H_i\subset V$ is given by the equation $x_i = 0$. 

$\CS$ is the corresponding stratification of $V$.

The strata of $\CS$ are in bijection with sequences 
\[s = (s_1, \ldots, s_n) \in \{-1, 0, 1\}^n,\]
where $C_s$ is given by $n$ equalities and inequalities: 
$x_i = 0$ if $s_i = 0$, and $s_ix_i >  0$ if $s_i = \pm 1$. So all the strata are cells, and in what
follows we will refer to the strata as cells.

To put it differently, the cells are numbered by couples $(I, \epsilon)$ where 
$I$ is a subset of $[n]:=\{1,\ldots,n\}$, and $\epsilon = (\epsilon_j )\in \{1, -1\}^{[n]\setminus I}$
is a collection of signs on the complement. We denote the corresponding cell by $C_{(I, \epsilon)}$.

We have $C_{(I, \epsilon)} \leq C_{(I', \epsilon')}$ (i.e.\ $C_{(I, \epsilon)}$ lies in the closure of
$C_{(I', \epsilon')}$) if and only if 

(a) $I \supset I'$, whence $\ol{I} := [n]\setminus I \subset \ol{I'}$,

and 

(b) $\epsilon_j = \epsilon'_j$ for all $j\in \ol{I}$.

In other words, $(I, \epsilon)$ is obtained from $(I,\epsilon')$ by replacing some elements
$\epsilon'_j$ in the sequence $\epsilon'$ by $0$'s. 

Consider two cells $C_s$ and $C' = C_{s'}$ where the sequence $s'$ is obtained 
from $s$ by replacing some element $s_i = \pm 1$ by $s'_i = - s_i$. We call such two cells
{\it neighbours}. Let $C'' = C_{s''}$ where $s''$ is obtained 
from $s$ by replacing the element $s_i$ by $0$. This $C''$ is called {\it the wall} between 
$C$ and $C'$. 

Following Tits, let us call {\it a gallery} a sequence
$C_0, \ldots, C_r$
where all $C_i, C_{i+1}$ are neighbours. We call $C_0$ and $C_r$ {\it faraway neighbours}. 

Each cell $C$ of dimension $p$ has $2^p$ faraway neighbours (including itself), whose union 
is dense in $L(C)$: {\it the flat } of $C$, the $\BR$-vector space spanned by $C$.  

\bigskip

We denote by $C_I$ the positive cell $C_{(I,\epsilon)}$ with all $\epsilon_j = 1$. 
So, there are $2^n$ positive cells.

\subsection{Complex coordinate arrangement}

Let \[\CH_\BC = \{H_1^\BC, \ldots, H_n^\BC\}\]
be the complex coordinate arrangement in $V_\BC = \BC^n$, and let
$\CS_{\BC}$ be the corresponding stratification of $V_\BC$.

Its strata are $S_I, I\subset [n] : = \{1, \ldots, n\}$, where
\[S_I = \{x\in V_\BC| x_i = 0 \text{\ for}\ i\in I, x_i\neq 0\text{\ for}\ i\notin I\}.\]

We write $S \leq S'$ iff $S\subset \ol{S'}$. 

In the complex case $S_I \leq S_J$ iff $J\subset I$.

A complex stratum $S_I$ with $|I| = p$ is of real codimension $2p$, and contains 
$2^{n-p}$ real cells $C_{(I,\epsilon)}$.

\subsection{GGM sheaves}
\label{2.2}
We have two linear algebra descriptions of the category 
$\CM(V_\BC, \CS_\BC) = \Perv(V_\BC, \CS_\BC)$. 

The first one is due to~\cite{ggm}. Namely, $\CM(V_\BC, \CS_\BC)$ is equivalent 
to  a category $\GGM_n$ whose objects are certain   
bisheaves $\Phi$ on the poset $\CS_\BC$ with values in $\Vect$, i.e.\

--- collections of $2^n$ spaces 
\[\Phi(I) = \Phi(S_I)\in \Vect,\ I\subset[n],\]

--- and {\em canonical} maps  
\[u_{IJ}\colon \Phi(I) \to \Phi(J),\ I\subset J,\]
and {\em variation} maps 
\[v_{JI}\colon \Phi(J) \to \Phi(I),\ I\subset J,\]
satisfying the transitivity for triples $I\subset J \subset K$.

They must satisfy the following further conditions.

For all $I, i\in I $ denote 
\[v_i=v_{I,I\setminus \{i\}}\colon \Phi(I) \leftrightarrows \Phi(I\setminus \{i\})\ : 
u_{I\setminus \{i\}, I}=u_i.\] 

Then the {\em mixed commutativity} condition: $uv = vu$ must hold, that is, for any
$i, j\in I, i \neq j$, the square 
\[\begin{CD}
\Phi(I\setminus \{i\}) @>{v_j}>> \Phi(I\setminus \{i, j\})\\
@V{u_i}VV  @V{u_i}VV \\
\Phi(I)  @>{v_j}>>  \Phi(I\setminus \{ j\})
\end{CD}\]
must commute. 

Also, the {\em invertibility} condition of $1+vu$ must hold, that is for all
$i\in I\subset[n]$, the {\em monodromy} operator
\[T_i:=1+v_iu_i\colon \Phi_{I}\to \Phi_{I}\]
must be invertible. 

Note that $1+v_iu_i$ is invertible if and only if $1+u_iv_i$ is invertible. 

This definition makes sense if we replace the coefficient category $\Vect$ by an 
arbitrary {\it additive} category $\CC$; let us denote the corresponding GGM category 
$\GGM_n(\CC)$.
%It is equivalent to the category of perverse sheaves
%$\CM_\CC(V_\BC,\CS_\BC)=\Perv_\CC(V_\BC,\CS_\BC)$ with coefficients in $\CC$.

\subsection{Fourier-Sato transform}
\label{2.2.1}
The {\em Fourier-Sato transform} $\Perv(V_\BC,\CS_\BC)\iso\Perv(V_\BC,\CS_\BC)$ induces
the same named auto-equivalence $\GGM_n\iso\GGM_n$. Note that in general, the Fourier-Sato
transform goes from perverse sheaves (smooth along a central hyperplane arrangement) on $V_\BC$
to perverse sheaves (smooth along the dual central hyperplane arrangement) on the dual space
$V^*_\BC$. However, in our special case of the coordinate hyperplane arrangement $\CS_\BC$,
the coordinates $x_1,\ldots,x_n$ are chosen on $V_\BC=\BC^n$, and we identify $V^*_\BC$ with
$V_\BC$ via the quadratic form $x_1^2+\ldots+x_n^2$. Then the dual hyperplane arrangement
coincides with the original $\CS_\BC$.

More generally, for an arbitrary additive
category $\CC$, we have an auto-equivalence \[\on{FT}\colon\GGM_n(\CC)\iso\GGM_n(\CC)\]
that sends an object $\Phi=(\Phi(I),u,v)$ to $\on{FT}(\Phi)=(\Phi'(I):=\Phi(\ol{I}),u',v')$.
Here we set for $i\not\in I\subset[n]$, 
\[u'_i:=-v_i\colon\Phi(\ol{I})\to\Phi(\ol{I}\setminus\{i\}),\ \operatorname{and}\
v'_i:=u_i(1+v_iu_i)^{-1}\colon\Phi(\ol{I}\setminus\{i\})\to\Phi(\ol{I}),\]
cf.~\cite[Proposition 4.5]{bk} and~\cite[(II.2.1), (II.3.2)]{m}.
  %$\on{FT}(\Phi) = (\Phi(\ol I), v,\pm u)$.
  %More precisely, $(\on{FT}u)_{IJ}=v_{\ol{I}\ol{J}}\colon\Phi(\ol{I})\to\Phi(\ol{J})$, and
  %$(\on{FT}v)_{JI}=(-1)^{|I|-|J|}u_{\ol{J}\ol{I}}\colon\Phi(\ol{J})\to\Phi(\ol{I})$.

\subsection{Hyperbolic sheaves}
The second, ``M\"obius dual'', linear algebraic description of $\CM(V_\BC, \CS_\BC)$
is a version of~\cite{ks1}. The category $\CM(V_\BC, \CS_\BC)$ 
is equivalent to a category $\Hyp_n$ of {\it hyperbolic sheaves}.

An object of $\Hyp_n$ is a 
 bisheaf $E$ with values in $\Vect$ on the poset of real strata $\CS$, i.e.

--- a collection of $3^n$ spaces $E(C) = E(I, \epsilon) \in \Vect,\ C\in \CS$,

--- maps $\gamma_{CC'}\colon E(C)\to E(C'),\ \delta_{C'C}\colon E(C)\to E(C'),\ C\leq C'$.

These maps must satisfy the following  conditions: 

%A VERSION THAT WE NEED: THE MOEBIUS DUAL TO GGM

{\em Idempotence}: For $C \leq C'$,
\[\gamma_{CC'}\delta_{C'C} = \Id_{E(C')}.\] 

Let $C, C'$ be any pair of cells. There exists a cell $C''\leq C$ and $C''\leq C'$, for example 
$C'' = \{0\}$. Define a map 
\[\phi_{CC'} = \gamma_{C''C'}\delta_{CC''}\colon E(C)\to E(C').\]
Due to the idempotence axiom this map does not depend on a choice of $C''$. 

{\em Invertibility}: If $\dim C = \dim C' = p$, and $C, C'$ belong to the same linear subspace
$L\subset V$ of dimension $p$ then $\phi_{CC'}$ is invertible. 

The invertibility axiom is equivalent to a seemingly weaker requirement:

{\em Weak invertibility}: The map $\phi_{CC'} = \gamma_{C''C'}\delta_{CC''}$ is an isomorphism for all
neighbours.

{\em Mixed commutativity}: $\gamma\delta = \delta\gamma$. That is,
for $C\leq D,\ C'\leq D',\ C\leq C',\ D\leq D'$, the square
\[\begin{CD}  E(C') @>{\delta}>> E(D') \\
@V{\gamma}VV @V{\gamma}VV \\
 E(C) @>{\delta}>> E(D) \\
\end{CD}\]
must commute.

This definition makes sense if we replace the coefficient category $\Vect$ by an 
arbitrary  category $\CC$; let us denote the corresponding category 
$\Hyp_n(\CC)$.

\subsection{Morita equivalence}
\label{2.3}
We have  equivalences of categories 
$L_h\colon  \CM(V_\BC, \CS_\BC) \iso\Hyp_n$ (``hyperbolic, or big localization'') and 
$L_g\colon  \CM(V_\BC, \CS_\BC) \iso\GGM_n$ (``vanishing cycles, or small localization''),
so there should exist equivalences of quiver categories 
\[P\colon\Hyp_n \iso\GGM_n,\ Q\colon\GGM_n \iso\Hyp_n.\]

Below we will construct equivalences
\begin{equation}\label{2.3.1}
  P\colon \Hyp_n(\CC) \iso\GGM_n(\CC),\ Q\colon \GGM_n(\CC) \iso\Hyp_n(\CC)
\end{equation}
for an arbitrary additive category $\CC$.

\subsection{Example $n = 1$}
\label{2.4}
The construction below is a version of~\cite[9.A]{ks1}.  

An object of $\Hyp_1(\CC)$ is a triple  $E = (E_0, E_\pm )$ of objects of $\CC$, and~$4$ maps  
\[\gamma_\pm\colon E_0 \to E_\pm,\ \delta_\pm\colon E_\pm \to E_0\]
such that $\gamma_\pm\delta_\mp$ are isomorphisms.

An object of $\GGM_1(\CC)$ is a couple $\Phi, \Psi\in (\CC)$ and~$2$ maps
\[v\colon \Phi\leftrightarrows \Psi\ : u\]
such that $t:=1+vu$ is an isomorphism.

Let $E \in\Hyp_1(\CC)$; we set $\Psi:= E_+,\ \Phi:=\Ker\gamma_-$, and $v$ is the composition
\[\Phi\hookrightarrow E_0\xrightarrow{\gamma_+} \Psi,\]
while $u$ is the composition
\[\Psi\xrightarrow{\delta_+}E_0\xrightarrow{1+\delta_-\gamma_-}\Phi.\]
This way we get an object
\[P(E) = (\Phi, \Psi, v, u)\in\GGM_1(\CC).\]

In the opposite direction, given $G = (\Phi, \Psi, v, u)\in\GGM_1(\CC)$ with 
$t:=1+vu$, we define 
\[Q(G) = (E_0, E_\pm, \gamma_\pm, \delta_\pm)\in\Hyp_1(\CC)\]
by 
\[E_+ = E_- = \Psi; \ E_0 = \Phi\oplus \Psi;\]
where $\delta_-$ (resp. $\gamma_-$) is given by the obvious inclusion (resp., projection), 
whereas 
\[\gamma_+(x, y) =-v(x) + y,\]
and
\[\delta_+(y) = (u(y), t(y)).\]

\subsection{General $n$}
\label{2.5}
We note that 
\[\GGM_n(\CC)\cong\GGM_1(\GGM_{n-1}(\CC)),\ \text{and}\ 
\Hyp_n(\CC)\cong\Hyp_1(\Hyp_{n-1}(\CC)).\]
Whence we define $P_n,Q_n$ by induction: $P_n$ being given as the composition
\begin{multline}
  \label{2.5.1}
  \Hyp_n(\CC)\cong\Hyp_1(\Hyp_{n-1}(\CC))
  \xrightarrow{P_1}\GGM_1(\Hyp_{n-1}(\CC))\\
\xrightarrow{P_{n-1}}\GGM_1(\GGM_{n-1}(\CC))\cong\GGM_n(\CC),
\end{multline}
and similarly for $Q_n$.

\subsection{Vanishing cycles at the origin: the Master complex}
\label{2.6}
We will give an explicit description of the equivalence $P\colon\Hyp_n \iso\GGM_n$
via vanishing cycles. Given $E = (E(C), \gamma, \delta)\in\Hyp_n$, 
we are going to describe the corresponding GGM-sheaf $P(E) = (\Phi(I),u,v)$.

First let us describe its vanishing cycles at the origin, $\Phi_0 := \Phi([n])$. 
Geometrically, let $M\in \CM(V_\BC, \CS_\BC)$ be the perverse sheaf corresponding to $E$; 
the sheaf of vanishing cycles $\Phi_f(M)$ corresponding to a function 
$f(x) = \sum_{i=1}^n x_i$ is a scyscraper concentrated at $0$, and $\Phi_0$ is its stalk at $0$:
\[\Phi_0 = \Phi_f(M)_0.\]
Algebraically the procedure of computing $\Phi_f(M)$ is described in~\cite{fks}, for an arbitrary 
real hyperplane arrangement. For the coordinate arrangement it reduces to the following.  

We consider the positive cells $C_I$, and denote $E(I) := E(C_I)$. Let 
$\CS^+ \subset \CS$ denote the subset of positive cells.

Let 
\[\on{Sub}_n^i = \{ I\subset[n]\ :\ |I| = n - i\} \cong \CS_i^+ := \{ C\in \CS^+|\ \dim C = i\}.\]
We consider a complex
\begin{equation}
  \label{(2.6.1)}
  \Phi^\bullet(E)\colon 0 \to E([n]) \to \oplus_{I\in \CS_1^+} E(I) \to \ldots 
\to E(\emptyset) \to 0, 
\end{equation}
%or, in a shorter notation
%\[\Phi^\bullet(E)\colon 0 \to E_\emptyset \to \oplus_{i\in [n]} E_i \to \oplus_{i < j} E_{ij} \to 
% \ldots \to E_{12\ldots n} \to 0,\]
concentrated in degrees $0, \ldots, n$. 

The differentials are $\pm \gamma$. More precisely, if $J=I\sqcup j$, then the $JI$ matrix coefficient
of the differential is $(-1)^s\gamma_{C(J)C(I)}$, where $s=|\{i\in I\ :\ i<j\}|$.

\begin{prop}
  \label{thm 2.6.1}
  The complex $\Phi^\bullet(E)$ is acyclic in positive degrees, and  
  \[\Phi_0 = H^0(\Phi^\bullet(E)).\]
\end{prop}

This is a particular case of~\cite[Theorem 2.3]{fks}. 
This complex $\Phi^\bullet(E)$ will be called the {\em Master complex of vanishing cycles} 
for a hyperbolic sheaf $E\in\Hyp_n$. 

\subsection{Explicit description of $P$}
Similarly, for an arbitrary $I\subset[n]$, $|I| = p$, let 
$f_I(x) = \sum_{i\in I} x_i$. The sheaf of vanishing cycles  
\[\Phi_{f_I}(M)\]
is supported on the closure of the  stratum $\overline{S(I)}$, and $\Phi(I)$ is the stalk 
of $\Phi_{f_I}(M)$ at $C_I$. It is computed via an exact sequence 
\[0 \to \Phi(I) \to E(I) \to \oplus_{I'\subset I, |I| = p - 1} E(I') \to \ldots \to 
E(\emptyset) \to 0,\]
where the differentials are $\pm \gamma$.

In particular, $\Phi(I)\subset E(I),\ \Phi(I\setminus\{i\})\subset E(I\setminus\{i\})$,
and the variation map $v_i\colon\Phi(I)\to\Phi(I\setminus\{i\})$ (notation of~\S\ref{2.2})
is induced by \[\gamma_{C_IC_{I\setminus\{i\}}}\colon E(I)\to E(I\setminus\{i\}).\]
Finally, the canonical map $u_i\colon\Phi(I\setminus\{i\})\to\Phi(I)$  is induced by the
composition \[E(I\setminus\{i\})\xrightarrow{\delta_{C_{I\setminus\{i\}}C_I}}E(I)
\xrightarrow{1+\delta_-\gamma_-}E(I),\]
where $E(I)=E(C_I)=E(C_{I,\epsilon})\stackrel{\gamma_-}{\longrightarrow}E(C_{I\setminus\{i\},\epsilon'})
\stackrel{\delta_-}{\longrightarrow}E(C_{I,\epsilon})=E(C_I)=E_I$, and $\epsilon'$ consists of all
positive signs except for one negative sign of $i$.

\begin{rem}
%\subsection{Explicit description of $Q$}
\label{2.7}
We will not give an explicit description of the equivalence $Q\colon\GGM_n \iso\Hyp_n$;
we just mention that for $\CG = (\Phi(I),u,v)$, the {\em positive} hyperbolic stalks of
the corresponding  $E = Q(\CG) = (E(C),\gamma,\delta)\in\Hyp_n$ are given by
\[E(I) = \oplus_{J\subset I} \Phi(J)\]
(``Takeuchi formula'', cf.~\cite[4.C.(1)]{ks1} and~\cite{t}).
%Other hyperbolic stalks are reconstructed by the $(\BZ/2\BZ)^n$-symmetry. 
\end{rem}

\subsection{Amonodromic semisimple sheaves}
\label{2.8}
Fix a semisimple abelian category $\CC$. The category of perverse sheaves
$\CM_\CC(V_\BC,\CS_\BC)=\Perv_\CC(V_\BC,\CS_\BC)$ with coefficients in $\CC$ contains a full semisimple
subcategory $\Perv^\ass_\CC(V_\BC,\CS_\BC)$ formed by the direct sums of constant sheaves along the
strata closures. Under the equivalences $\Perv_\CC(V_\BC,\CS_\BC)\simeq\Hyp_n(\CC)$ and
$\Perv_\CC(V_\BC,\CS_\BC)\simeq\GGM_n(\CC)$ the subcategory
$\Perv^\ass_\CC(V_\BC,\CS_\BC)\subset\Perv_\CC(V_\BC,\CS_\BC)$ corresponds
to the full semisimple subcategories $\Hyp^\ass_n(\CC)\subset\Hyp_n(\CC)$ and
$\GGM^\ass_n(\CC)\subset\GGM_n(\CC)$.

Namely, $\GGM^\ass_n(\CC)\subset\GGM_n(\CC)$ consists of all GGM sheaves $(\Phi(I),u,v)$ as
in~\S\ref{2.2} with all $u_{IJ} = v_{JI} = 0$.

The description of $\Hyp^\ass_n(\CC)\subset\Hyp_n(\CC)$ is a little bit more elaborate.
Given two cells $C_{(I,\epsilon)}\leq C_{(I',\epsilon')}$ as in~\S\ref{2.1}, we define the
{\em reflected} cell $C_{(I'',\epsilon'')}$ as follows. Recall from~\S\ref{2.1} that $(I,\epsilon)$ is
obtained from $(I',\epsilon')$ be replacing some elements $\epsilon'_j$ in the sequence $\epsilon'$
by $0$'s. Now to obtain $(I'',\epsilon'')$ we replace these elements by the opposite ones (i.e.\
change their signs). In particular, $I''=I'$, and $\dim C_{(I',\epsilon')}=\dim C_{(I'',\epsilon'')}$.

Finally, a hyperbolic sheaf $(E(C),\gamma,\delta)$ lies in $\Hyp^\ass_n(\CC)$ if for any cells
$C_{(I,\epsilon)}\leq C_{(I',\epsilon')}$ and the corresponding reflected cell $C_{(I'',\epsilon'')}$ we have
\begin{equation}
  \label{Hyp ass}
  \Ker\gamma_{C_{(I,\epsilon)}C_{(I',\epsilon')}}=\Ker\gamma_{C_{(I,\epsilon)}C_{(I'',\epsilon'')}}\ \on{and}\
  \on{Im}\delta_{C_{(I',\epsilon')}C_{(I,\epsilon)}}=\on{Im}\delta_{C_{(I'',\epsilon'')}C_{(I,\epsilon)}}.
\end{equation}

The equivalence $P\colon\Hyp^\ass_n(\CC)\iso\GGM^\ass_n(\CC)$ takes $(E(C),\gamma,\delta)$ to
$(\Phi(I),0,0)$, where $\Phi(I):=\bigcap_{J\subsetneqq I}\Ker\gamma_{C_IC_J}$.

The inverse equivalence $Q\colon\GGM^\ass_n(\CC)\iso\Hyp^\ass_n$ takes $(\Phi(I),0,0)$ to
$(E(C),\gamma,\delta)$, where $E(C_{(I,\epsilon)}):=\bigoplus_{J\subset I}\Phi(J)$, while $\gamma$ are
given by the natural embeddings of direct summands, and $\delta$ are given by the natural projections
onto direct summands.

We may say that these two descriptions of the amonodromic semisimple sheaves are
M\"obius dual to each other.

\begin{rem}
The autoequivalence of~\S\ref{2.2.1}
\[\on{FT}\colon\GGM_n(\CC) \iso\GGM_n(\CC)\]
induces an involutive autoequivalence  
\begin{equation}
  \label{2.8.2}
  \on{FT}\colon\GGM^\ass_n(\CC) \iso\GGM^\ass_n(\CC),\ (\Phi(I),0,0)\mapsto(\Phi(\ol{I}),0,0).
\end{equation}
The corresponding involutive autoequivalence
\begin{equation}
  \label{2.8.3}\on{FT}\colon\Hyp^\ass_n(\CC) \iso\Hyp^\ass_n(\CC)
\end{equation}
is a particular case of~\cite[Theorem 4.15]{fks}.
\end{rem}

\section{Kostka numbers}
In the next section we will study the categories of amonodromic semisimple perverse sheaves
$\Hyp^\ass_n(\CC),\ \GGM^\ass_n(\CC)$ for the category $\CC=\Rep(S_{n+1})$ of representations of
the symmetric group $S_{n+1}$ 
on $n+1$ letters over a field $\sk$ of characteristic~0.

In this section we prepare some basic material about 
Young tableaux, Kostka numbers, and $\Rep(S_{n+1})$.
We also discuss a generalization of Kostka numbers 
to a finite Coxeter group $W$.
 
%In this section, we will sometimes use $m=n+1$.

\subsection{Partitions and compositions}
\label{1.1}
A {\em partition} $\lambda$ of $m$ is a sequence
of integers 
\[
\lambda = (\lambda_1, \ldots, \lambda_k),\ 
\lambda_1\geq \ldots \geq \lambda_k>0,\
\sum \lambda_i = m. 
\]
%The number $\ell(\lambda):=k$ is the {\em length} of $\lambda$.
%We assume that $\lambda_i=0$ for $i>k$.
The {\em Young diagram} of a partition $\lambda$ is the set
$\{(i,j)\in\mathbb{Z}^2\mid 1\leq i\leq k,\ 1\leq j\leq \lambda_i\}$.
The elements $(i,j)$ of a Young diagram, called {\em boxes}, 
are arranged on the plane and indexed as elements
of a matrix, i.e., $i$ is the row index of box $(i,j)$ 
and $j$ is its column index.
Slightly abusing notation, we identify a partition $\lambda$ 
with its Young diagram and denote the latter by the same letter $\lambda$.
For example, we have 
\[ 
\lambda = (3,3,1) = 
\ydiagram{3,3,1}
\]

Let $\fP_{m}$ be the set of partitions of $m$, equivalently, the set of Young diagrams with $m$ boxes, 
and let
$\fP := \bigcup_{m\geq 0} \fP_{m}$.
%In particular, $\fP_0 = \{\emptyset\}$ consists of a single Young diagram $\emptyset$.

A {\em composition} $\beta$ of $m$ is a sequence of positive integers
\[
\beta = (\beta_1, \ldots , \beta_k)
\textrm{ with }
\sum_i \beta_i = m.
\]
Let $\fQ_{m}$ be the set of all compositions of $m$.
%and let $\fQ := \bigcup_{m\geq 0} \fQ_{m}$.
We have an obvious inclusion  
\begin{equation}
  \label{(1.1.1)}
  \iota\colon \fP_m\hookrightarrow \fQ_m
\end{equation}
with a left inverse $\fQ_m\to \fP_m$,
that associates to a composition its weakly decreasing permutation.

\subsection{Young tableaux}
%For $\lambda, \mu\in \fP_{n+1}$ we write $\lambda \geq \mu$ iff
%\[\lambda_1 + \ldots + \lambda_j \geq \mu_1 + \ldots + \mu_j\]
%for all $j$. It is a partial order on $\fP_n$.

%Let $\lambda\in \fP_n$. 
For a Young diagram $\lambda\in\fP$, 
a {\em semi-standard Young tableau} $T$ of {\em shape} 
$\lambda$ is a filling of boxes of $\lambda$ by positive integers 
such that the entries are weakly increasing 
along each row and strictly increasing along each column of $T$.
The {\em weight} of $T$ is a nonnegative integer sequence
$\beta = (\beta_1, \ldots , \beta_k)$, 
where $\beta_i$ is the number of $i$'s in $T$.
(Here we identify weights $\beta$ obtained from each other by adding 0's at the end:
$\beta=(\beta_1,\dots,\beta_k) = (\beta_1,\dots,\beta_k,0,\dots,0)$.)
A {\em standard Young tableau} $T$
a semi-standard Young tableau 
of weight $\beta = (1, \ldots, 1)$.
%of shape $\lambda/\mu$ is 
%$T$ of shape
%$\lambda/\mu$ 
We denote by $\on{SSYT}(\lambda,\beta)$  the set
of all semi-standard Young tableaux of shape $\lambda$ and weight $\beta$.  
Also $\on{SYT}(\lambda)$ denotes the set
of all standard Young tableaux of shape $\lambda$.  
For example, here is a semi-standard Young tableau $T \in\on{SSYT}((4,3,1),(2,2,3,1))$ 
and a standard Young tableau
$T'\in\on{SYT}((4,3,1))$:
\[
T=
\begin{ytableau}
1 & 1 & 2 & 3 \\
2 & 3 & 3 \\
4
\end{ytableau}
\qquad
\qquad
T' =
\begin{ytableau}
1 & 2 & 4 & 7 \\
3 & 5 & 6 \\
8
\end{ytableau}
\]

Clearly, $\#\on{SSYT}(\lambda,\beta) = \#\on{SSYT}(\lambda,\beta')$,
where $\beta'$ is the composition obtained from weight $\beta$ by removing all 0's.

\subsection{Kostka numbers}

For a partition $\lambda$ and a composition $\beta$, the {\em Kostka number} 
$K_{\lambda,\beta}$ is defined as
the number of semi-standard Young tableaux of shape $\lambda$ 
and weight $\beta$: 
\[
K_{\lambda,\beta} := \# \on{SSYT}(\lambda,\beta).
\]

%It is known that 
The Kostka numbers $K_{\lambda,\beta}$ 
are known to be
invariant under permutations of parts of $\beta$.
Thus, without loss of generality,  
the Kostka numbers can be labelled by a pair of partitions
$\lambda, \mu \in \fP$, 
$K_{\lambda\mu} = K_{\lambda, \iota(\mu)}$, 
where $\iota$ denotes the inclusion~\eqref{(1.1.1)}. 
If $\mu = \lambda$, the only semi-standard Young tableau of the same shape and weight 
$\lambda$ is obtained by filling the $i$th row of $\lambda$ all all $i$'s, for $i=1,2,\dots$.
%a Young diagram in a semi-standard manner is to put all $1$'s in the first row, 
%all $2$'s in the second row, etc., 
So $K_{\lambda\lambda} = 1$.

\subsection{Small Kostka numbers} 
\label{1.4.1}

Let us fix a nonnegative integer $n$.  Let $[n]:= \{1, \ldots, n\}$,
and let $\on{Sub}_n:=2^{[n]}$ denote the set of subsets $I\subset [n]$.
The set $\on{Sub}_n$ is in bijection 
with the set $\fQ_{n+1}$ of compositions of $n+1$.

%Explicitly, define the two maps $\varrho=\varrho_n\colon\on{Sub}_n 
%\overset{\sim}\to \fQ_{n+1}$
%and $\varpi = \varpi_n \colon\fQ_{n+1} \iso  \on{Sub}_n$
%by 

\begin{lem}
  \label{1.1.1}
The following two maps
$\varrho=\varrho_n\colon\on{Sub}_n 
\overset{\sim}\to 
\fQ_{n+1}$
and 
 $\varrho^{-1} 
%= \varrho_n^{-1} 
\colon\fQ_{n+1} 
\overset{\sim}\to 
\on{Sub}_n$
are bijective and inverse to each other:
%and $\varpi$ are bijective and inverse to each other.
\[
\begin{array}{l}
\varrho: 
I=\{i_1<i_2<\ldots<i_r\} \mapsto
\beta=
(i_1,i_2-i_1,\ldots,i_r-i_{r-1},n+1-i_r),\\[.1in]
%\]
%\[
\varrho^{-1}: \beta= (\beta_1, \ldots, \beta_k)\mapsto
I=\{\beta_1,  \beta_1 + \beta_2,  \ldots, 
\beta_1 + \ldots + \beta_{k-1} \}.
\end{array}
\]

\end{lem}

For example, $\varrho([n]) = (1,1,\ldots,1)$ and $\varrho(\emptyset) = (n+1)$.

\begin{defn}
  \label{1.4.2}
  For $\lambda\in \fP_{n+1}$ and $I\in\on{Sub}_n$, the {\em small Kostka number} 
$\kappa_{\lambda,I}$ 
is defined by 
\[\kappa_{\lambda, I} := \sum_{J \subset I} (-1)^{|J|-|I|} K_{\lambda,\varrho(J)}.\] 
\end{defn}

\begin{prop}
  \label{1.4.3} 
For $\lambda\in\fP_{n+1}$ and 
$I\in\on{Sub}_n$,
% and $\beta\in\fQ_{n+1}$,
we have 
  \begin{equation}
    \label{1.4.3.1b}
K_{\lambda,\varrho(I)}=\sum_{J\subset I}\kappa_{\lambda,J},
\end{equation}
or, equivalently,
%\begin{equation}
%  \label{1.4.3.1b}
  $K_{\lambda,\beta}=\sum_{J\subset\varrho^{-1}(\beta)}\kappa_{\lambda,J}$.
%\end{equation}
This formula defines the numbers $\kappa_{\lambda,I}$ uniquely.
\end{prop}

\begin{proof} The claim follows from the M\"obius inversion formula.
\end{proof}

\begin{example}  
Let $n=2$ and $\lambda=(2,1)\in\fP_3$.
We have $K_{\lambda,\rho(\emptyset)} = K_{\lambda,(3)} = 0$,
$K_{\lambda,\rho(\{1\})} = K_{\lambda,(1,2)}= 1$,
$K_{\lambda,\rho(\{2\})} = K_{\lambda,(2,1)}= 1$,
$K_{\lambda,\rho(\{1,2\})} = K_{\lambda,(1,1,1)}= 2$.
Thus the small Kostka numbers are
$\kappa_{\lambda,\emptyset} = 0$,
$\kappa_{\lambda,\{1\}} = 1-0 = 1$,
$\kappa_{\lambda,\{2\}} = 1-0 = 1$,
$\kappa_{\lambda,\{1,2\}} = 2-1-1+0 = 0$.
\end{example}

\begin{rem}
  For $n\leq 4$, all small Kostka numbers $\kappa_{\lambda,I}$ are either~0 or~1. 
  For $n\geq 5$, the numbers $\kappa_{\lambda,I}$ can be greater than 1.
  For example, for $n=3k-1$, we have
$\kappa_{(2k-1,k,1),\{k,2k\}}=k$. 
\end{rem}

%\begin{rem}
%Since subsets $I\in\on{Sub}_n$ are in bijection with compositions $\beta\in \fQ_{n+1}$, we can
%also write the small Kostka numbers as $\kappa_{\lambda,\varrho^{-1}(\beta)}$.
%However, in general, 
%the small Kostka numbers  $\kappa_{\lambda,\varrho^{-1}(\beta)}$ 
%(unlike the usual Kostka numbers $K_{\lambda,\beta}$),
%are not invariant under 
%permutations of parts of $\beta$.
%So we cannot label the small Kostka numbers by a pair of partitions.
%\end{rem}

For $\lambda\in\fP_{n+1}$, let $\lambda^t$ denote the {\em conjugate partition} 
of $\lambda$, i.e., the partition whose Young diagram is obtained 
from $\lambda$ by transposition. 
For $I\in \on{Sub}_n$, let $\ol{I} := [n]\setminus I$.

\begin{thm}
  \label{th:small_Kostka_nonneg_symm}
{\rm (1)}
The small Kostka numbers $\kappa_{\lambda,I}$ are non-negative integers.
\smallskip

{\rm (2)} 
The small Kostka numbers have the symmetry
$\kappa_{\lambda,I} = \kappa_{\lambda^t,\ol{I}}$.
\end{thm}

This theorem follows from a combinatorial and a representation-theoretical
interpretations of the small Kostka numbers $\kappa_{\lambda, I}$ given below, 
see Theorem~\ref{1.6} and Corollary~\ref{th:small_kostka=multiplicity}.

\subsection{Descents in Young tableaux}

For a standard Young tableau $T\in\on{SYT}(\lambda)$ of shape 
$\lambda\in \fP_{n+1}$,   we say that  $i\in [n]$ is a {\em descent} of $T$ 
if the entry $i+1$ is located in $T$ below the entry $i$, i.e., 
the row containing $i+1$ is below the row containing $i$. 
The {\em descent set} $\on{Des}(T) := \{i\in [n]\mid i\textrm{ is a descent of } T\} \in \on{Sub}_n$, the set of all descents of $T$. 

Let $T^t$ denote the standard Young tableau obtained by transposing $T$.
Clearly, the descent set of $T^t$ is the complement 
to the descent set of $T$:
$\on{Des}(T^t) = [n]\setminus \on{Des}(T)$.  For example, for  
\[T =
\begin{ytableau}
1 & 2 & 4 & 8 \\
3 & 5 & 7\\
6
\end{ytableau}
\qquad
\textrm{and} 
\qquad
T^t =
\begin{ytableau}
1 & 3 & 6\\
2 & 5\\
4 & 7 \\
8
\end{ytableau}
\]
we have $\on{Des}(T) = \{2, 4, 5\}$
and $\on{Des}(T^t) = [7]\setminus\{2,4,5\}=
\{1,3,6,7\}$.

Theorem~\ref{th:small_Kostka_nonneg_symm} follows immediately from
the following claim.

\begin{thm}
  \label{1.6}
  The small Kostka number $\kappa_{\lambda,I}$ equals the number of 
standard Young tableaux  of shape $\lambda$ with descent set
  $\on{Des}(T) = I$.
\end{thm}

\begin{proof}

Define the {\em standardization map}
\[
\on{std}:
\on{SSYT}(\lambda,\beta)\to 
\on{SYT}(\lambda) 
\]
as follows.
For a semi-standard Young tableau $T$ 
of weight $\beta = (\beta_1,\dots,\beta_k)$,
the {\em standardization} $\on{std}(T)$ of $T$
is the standard Young tableau obtained from $T$ by replacing 
its $\beta_i$ entries $i$ by the integers 
in the interval $[\beta_1+\cdots+\beta_{i-1} + 1, \beta_1+\cdots+\beta_i]$
%$\left(\sum_{i'<i} \beta_{i'}\right) + r$, for $r\in\{1,2,\dots,\beta_i\}$, 
arranged in the increasing order from left to right, for $i=1,\dots,k$.
For example, we have
%a semi-standard Young tableau $T \in\on{SSYT}((4,3,1),(2,2,3,1))$ 
%and its standardization
%$\on{std}(T)\in\on{SYT}((4,3,1)) $
\[
T=
\begin{ytableau}
1 & 1 & 2 & 3 \\
2 & 3 & 3 \\
4
\end{ytableau}
\quad
\longmapsto
\quad
\on{std}(T) =
\begin{ytableau}
1 & 2 & 4 & 7 \\
3 & 5 & 6 \\
8
\end{ytableau}
\]

%Let $T'$ be a semi-standard Young tableau of shape $\lambda$ and weight $\beta$.    
%Let us make a standard Young tableau $T$ from $T'$ by replacing $\beta_1$ 
%entries of~$1$ with
%$1,2,...,\beta_1$ (ordered from left to right), 
%then replacing $\beta_2$ entries of~$2$ with
%$\beta_1 + 1 , \beta_1 + 2, ..., \beta_1 + \beta_2$ 
%(also ordered from left to right), etc.
%This procedure, which does not change the shape, is called the standardization of a
%semi-standard tableau.

The standardization map gives a bijection between the set
$\on{SSYT}(\lambda,\beta)$ of all semi-standard Young tableaux $T$ of shape
$\lambda$ and weight $\beta$  and the subset of standard Young tableaux $T'\in
\on{SYT}(\lambda)$ that have no descents in the union of intervals
\[
[1,\beta_1 - 1] \cup 
[\beta_1 + 1,\beta_1 + \beta_2 - 1] \cup \cdots \cup
[\beta_1+\cdots+\beta_{k-1}+1,\beta_1+\cdots+\beta_k-1],
\]
or, equivalently, $\on{Des}(T')\subset \varrho^{-1}(\beta) = 
\{\beta_1,\beta_1+\beta_2,\dots,\beta_1+\beta_2+\cdots+\beta_{k-1}\}$:
\[
\on{std} \colon \on{SSYT}(\lambda,\beta)\iso
\{T'\in \on{SYT}(\lambda) \mid \on{Des}(T')\subset \varrho^{-1}(\beta)\}.
%= \on{SYT}(\lambda)\cap\on{Des}^{-1}(\varpi(\beta)).
\]

Let $\tilde\kappa_{\lambda,I}$ be the number of standard Young tableaux of 
shape $\lambda$ with descent set $I$.
We obtain
\[
K_{\lambda,\beta} = \sum_{J\subset \varrho^{-1}(\beta)} \tilde \kappa_{\lambda,J},
\]
which is exactly equation~\eqref{1.4.3.1b} that defines the small Kostka 
numbers.
Thus $\kappa_{\lambda,I} = \tilde\kappa_{\lambda,I}$, as needed.
\end{proof}

%the set 
%of standard Young tableaux $T$ such that the descent set 
%$\on{Des}(T)$ is contained in $\varpi(\beta)$. 
%

%This means that the Kostka number 
%$K_{\lambda,\beta}$ equals the number of SYT's $T$ of shape $\lambda$
%such that $\on{Des}(T) \subset \varpi(\beta).$

%\begin{cor}
%  \label{1.5}
%  We have $\kappa_{\lambda,I}\geq0$ for any $\lambda,I$. \hfill $\Box$
%\end{cor}

%\begin{lem}
%  \label{symmetry}
%  We have $\kappa_{\lambda,I}=\kappa_{\lambda^t,\ol{I}}$ for any $\lambda,I$.
%\end{lem}

%\begin{proof}
%Given a standard Young tableau $T$ of shape $\lambda$, we consider its transposed standard Young
%tableau $T^t$ of shape $\lambda^t$. It is immediate to check that if the descent set of
%$T$ is $I$, then the descent set of $T^t$ is $\ol{I}$.
%\end{proof}

\subsection{Skew Young diagrams}

For two Young diagrams $\lambda, \mu \in\fP$ such that $\lambda \supseteq \mu$, 
i.e., $\lambda$ contains $\mu$ as a subset, the {\em skew Young
diagram} $\lambda/\mu$ is the set-theoretic difference of the Young diagrams
$\lambda$ and $\mu$.
%We consider skew Young diagrams up to translations.  
For example, for $\lambda = (3,3,1)$ and $\mu=(2,1)$, 
%$\nu=(4,4,4,2)$, $\delta=(4,3,2,1)$, 
we have
\[
\lambda/\mu = 
%\nu/\delta=
\ydiagram{2+1,1+2,1}
\]
Usual Young diagrams can be considered as a special case of skew Young diagrams
$\lambda/\mu$ with $\mu=\emptyset$.

A {\em connected component} of a skew Young diagram $\lambda/\mu$ is a connected 
component of the graph on the set
of boxes $(i,j)\in \lambda/\mu$ with the edge set 
$\{((i,j),(i',j')) \mid |i-i'|+|j-j'|=1\}$.
For example, the above skew Young diagram has two connected components.

We consider two special types of skew Young diagrams: horizontal strips and
ribbons.\footnote{Ribbons appear in the literature under many different names.
They are MacMahon's zigzag diagrams, see
\cite{Mac}.  They also called rim hooks, border strips, etc.}
A skew Young diagram $\lambda/\mu$ is a {\em horizontal strip} if
each connected component of $\lambda/\mu$ consists of a single row of boxes.  A skew
Young diagram is a {\em ribbon} if it contains no $2\times 2$ square and
has exactly one connected component.

For a composition $\beta=(\beta_1,\dots,\beta_k)$, let 
$\on{hstrip}(\beta)$ be the horizontal strip 
whose $i$th row contains $\beta_i$ boxes, for $i=1,\dots,k$ (with rows touching each
other at the corners).
Also let $\on{ribbon}(\beta)$ be the ribbon whose $i$th row contains $\beta_i$ boxes,
for $i=1,\dots,k$.  For example,
\[
\on{hstrip}((2,3,1,2))=
\kern-.3in
\ydiagram{6+2,3+3,2+1,2}\ ,
\qquad
\on{ribbon}((2,3,1,2)) = 
\kern-.15in
\ydiagram{3+2,1+3,1+1,2}
\]

\subsection{Young symmetrizers}
Let $\sk[S_{n+1}]$ be the group algebra of the symmetric group $S_{n+1}$. 
Let $\lambda/\mu$ be a skew Young diagram with $n+1$ boxes.  Let $T$ be any
filling of boxes of $\lambda/\mu$ by 
$1,\dots,n+1$ (without repetitions), i.e., $T$ is any bijection 
$[n+1] \overset{\sim}\to \text{\{boxes of
$\lambda/\mu$\}}$.  (Here $T$ is not necessarily a standard Young tableau.) 
Let $R(T)$ (resp., $C(T)$) denote the subgroup of $S_{n+1}$ consisting of all permutations 
preserving the rows (resp., the columns) of $T$. 

The {\em Young symmetrizer} $y_T\in \sk[S_{n+1}]$ is defined as
\[a_T = 
\sum_{ w \in R(T)} w \, ,
\quad  
b_T = \sum_{w\in C(T)} \epsilon(w)\, w\,,
\textrm{ \  and \  }
y_T= b_T\, a_T\,.\]
Here $\epsilon(w)\in\{1,-1\}$ denotes the sign of permutation $w\in S_{n+1}$.

The left ideal 
\begin{equation}
  \label{(1.1.1)'}
  L_T := \sk[S_{n+1}]\cdot y_T\subset \sk[S_{n+1}]
\end{equation} 
is a representation of $S_{n+1}$.

If $T$ and $T'$ have the same shape $\lambda/\mu$,  
then $y_T = y_{T'} \cdot u$, for some $u\in S_{n+1}$.  Thus
$L_{T} \cong L_{T'}$ are isomorphic $S_{n+1}$-modules.

The {\em Specht module} $V_{\lambda/\mu}$ is defined (up to an isomorphism) as
\[
V_{\lambda/\mu}:=L_T,
\] 
for any filling $T$ of the shape $\lambda/\mu$.  A choice of $T$ gives an
embedding of the Specht module $V_{\lambda/\mu}$ 
as a left submodule $L_T$ of the group algebra $\sk[S_{n+1}]$.

For $\mu=\emptyset$, the Specht modules $V_\lambda$,
$\lambda\in\fP_{n+1}$, are exactly (the isomorphism classes of)
all irreducible representations of the symmetric group $S_{n+1}$.
We have the following decomposition of the group algebra $\sk[S_{n+1}]$
into a direct sum of left submodules:
\begin{equation}
  \label{1.1.2}
  \sk[S_{n+1}] = \bigoplus_{T\in \on{SYT}_{n+1}} L_T
=
\bigoplus_{\lambda\in\fP_{n+1}}\left(\bigoplus_{T\in\on{SYT}(\lambda)} L_T\right),
\end{equation}
where $\on{SYT}_{n+1}=\bigcup_{\lambda\in\fP_{n+1}} \on{SYT}(\lambda)$, 
%is the set of all standard Young tableaux of all shapes
%$\lambda$ with $n+1$ boxes, 
see~\cite[Theorem 1.3.J]{we}.

\subsection{Induced modules and ribbon modules}
\label{ssec:ind_rib}

The {\em semitic filling} of a skew Young diagram $\lambda/\mu$ 
is the filling $T$ of boxes of $\lambda/\mu$ by $1,2, 3, \dots$ reading the boxes
by rows right-to-left top-to-bottom.  For example, here are the semitic fillings
of the horizontal strip $\on{hstrip}((2,3,1,2))$ and the ribbon
$\on{ribbon}((2,3,1,2))$:
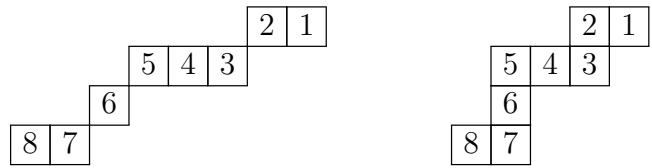
\[
\begin{ytableau}
\none & \none & \none & \none & \none & \none & 2 & 1 \\
\none & \none & \none & 5  & 4 & 3 \\
\none & \none & 6 \\
8 & 7 
\end{ytableau}
\qquad
\qquad
\begin{ytableau}
\none & \none & \none & 2 & 1 \\
\none & 5  & 4 & 3 \\
\none & 6 \\
8 & 7 
\end{ytableau}
\]

For a 
%subset $I\in\on{Sub}_n$ and the associated 
composition $\beta=(\beta_1,\dots,\beta_k)$ of $n+1$,
define the {\em induced module} 
$M_\beta := L_T$, where $T$ is the semitic filling of the horizontal strip
$\on{hstrip}(\beta)$.  Equivalently,
\begin{equation}
M_\beta := \sk[S_{n+1}]\cdot 
\left(\sum_{w\in S_{\beta_1}\times \cdots \times S_{\beta_k}} w \right)
\cong \Ind_{S_{\beta_1}\times \cdots \times S_{\beta_k}}^{S_{n+1}} \sk,
\end{equation}
where the direct product of groups 
$S_{\beta_1}\times \cdots \times S_{\beta_k}$ is standardly embedded as 
a subgroup of $S_{n+1}$.
For instance, $M_{(1,\dots,1)} = \sk[S_{n+1}]$ is the regular representation, 
and $M_{(n+1)} = \sk$ is the trivial representation of $S_{n+1}$.

It is well-known, cf.~\cite{ko} and~\cite[Corollary 4.39]{fh}, that 
the Kostka number $K_{\lambda,\beta}$ is exactly 
the multiplicity of the irreducible $S_{n+1}$-module $V_\lambda$
in the induced module $M_\beta$:
\begin{equation}
\label{eq:Kostka=mult}
K_{\lambda,\beta} = \dim\Hom_{S_{n+1}}(M_\beta,V_\lambda)\,,
\quad
\textrm{ equivalently, }
\quad
M_\beta\cong \bigoplus_{\lambda\in\fP_{n+1}} V_\lambda^{\oplus K_{\lambda,\beta}}.
\end{equation}

Define the {\em ribbon module} $R_\beta$ as 
\begin{equation}
\label{eq:ribbon_mod_def}
R_{\beta}:=L_{T'},
\end{equation} 
where $T'$ is 
the semitic filling of $\on{ribbon}(\beta)$.

Recall the bijection $\varrho:I\mapsto \beta$ between 
subsets $I\in\on{Sub}_n$ and compositions $\beta\in\fQ_{n+1}$, 
see Lemma~\ref{1.1.1}.   We can also label induced modules
and ribbon modules by subsets $I\in\on{Sub}_n$:
\begin{equation}
\label{eq:ind_rib_I}
M_I:=M_{\varrho(I)}
\qquad
\textrm{and}
\qquad
R_I:=R_{\varrho(I)}.
\end{equation}

Solomon~\cite[Theorem~2]{so2} constructed a decomposition of the group algebra
of a finite Coxeter group, see Theorem~\ref{th:solomon_dec_W} below.  
For the symmetric group $S_{n+1}$, this decomposition can be described as follows. 

%The following result 
%decomposition of the group algebra $\sk[S_{n+1}]$ into a direct
%sum of ribbon modules 
%is due to Solomon, see \cite[Theorem~2]{so2} and

\begin{thm}
\cite[Theorem~2]{so2} \  
The group algebra $\sk[S_{n+1}]$ decomposes 
into a direct sum of ribbon modules:
\begin{equation}
\label{eq:Solomon_dec_S}
\sk[S_{n+1}] = \bigoplus_{I\in\on{Sub}_n} R_I.
\end{equation}
\end{thm}
%We call this decomposition {\em Solomon's decomposition.}

More generally, any induced module $M_I$ decomposes 
into a direct sum of ribbon modules.
%Recall that we assume that $m=n+1$.  Also 
%Recall the bijection $\varrho:
%\on{Sub}_n \overset{\sim}\to \fQ_{n+1}$ 
%between 
%subsets $I\subset [n]$ and compositions $\beta$ of $n+1$, see Lemma~\ref{1.1.1}.

%This theorem follows from the following claim, which is a special case of 
%Solomon's \cite[Theorem~2]{so2}, see Theorem ??? below.

\begin{thm}
\label{th:ind_decom_ribbon}
For $I\in\on{Sub}_n$, the induced $S_{n+1}$-module $M_{I}$ decomposes into 
a direct sum of ribbon modules:
\begin{equation}
\label{eq:decomp_induced_ribbon}
M_{I} = \bigoplus_{J\subset I} R_{J},
\end{equation}
\end{thm}
This theorem is a special case of Corollary~\ref{cor:M_I=sum_of_R_I_for_W} below.
Note that here we treat all modules 
as concrete left submodules of the group algebra $\sk[S_{n+1}]$
(and not as equivalence classes of representations of $S_{n+1}$).

%\begin{proof}
%Let $\beta=\varrho(I)$.  Apply the projector $p:\sk[S_{n+1}]\to M_\beta$ given 
%by $p: f \mapsto f \cdot 
%\left(
%\sum_{w\in S_{\beta_1}\times \cdots \times S_{\beta_k}}w\right)$
%to all terms of Solomon's decomposition~\eqref{eq:Solomon_dec_S}.
%\end{proof}

\begin{cor}
\label{th:small_kostka=multiplicity}
For a partition $\lambda\in\fP_{n+1}$ and a subset $I\in\on{Sub}_n$,
the small Kostka number $\kappa_{\lambda,I}$ is the multiplicity
of the irreducible $S_{n+1}$-module $V_\lambda$ in the ribbon module
$R_{I}$:
\[
\kappa_{\lambda,I} = \dim\Hom_{S_{n+1}}(R_I, V_\lambda)\,,
\quad
\textrm{ equivalently, }
\quad
R_I \cong \bigoplus_{\lambda\in\fP_{n+1}} V_\lambda^{\oplus 
\kappa_{\lambda,I}}.
\]
\end{cor}

\begin{proof}
Let $\tilde \kappa_{\lambda,I}$ be the multiplicity of $V_\lambda$
in $R_I$.   Taking multiplicities of $V_\lambda$ in all terms of 
the decomposition~\eqref{eq:decomp_induced_ribbon}, we get
$K_{\lambda,\varrho(I)} = \sum_{J\subset I} \tilde \kappa_{\lambda,J}$,
which is exactly equation~\eqref{1.4.3.1b} that defines the small Kostka 
numbers.
Thus $\kappa_{\lambda,I} = \tilde\kappa_{\lambda,I}$, as needed.
\end{proof}

\begin{rem}
For any skew Young diagram $\lambda/\mu$,
the multiplicities of irreducible representations $V_\nu$ of the symmetric group
in the Specht module $V_{\lambda/\mu}$ 
are called the {\em Littlewood-Richardson coefficients} $c^\lambda_{\mu\nu}$: 
\[
V_{\lambda/\mu}\cong\bigoplus_{\nu} V_\nu^{\otimes c^\lambda_{\mu\nu}}. 
\]
The Littlewood-Richardson rule is a combinatorial rule for the $c^\lambda_{\mu\nu}$. 
%It is not hard to see that 
The combinatorial interpretation of the small Kostka numbers
(Theorem~\ref{1.6}) can be shown to be equivalent 
to the special case of the Littlewood-Richardson rule 
when $\lambda/\mu$ is a ribbon.
\end{rem}

%We can combine the combinatorial (Theorem~\ref{1.6}) 
%and the representation-theoretic (Corollary~\ref{th:small_kostka=multiplicity})
%interpretations of the small Kostka numbers $\kappa_{\lambda, I}$
%into the following result.

Theorem~\ref{1.6}, decomposition~\eqref{1.1.2}, and
Corollary~\ref{th:small_kostka=multiplicity} imply the following claim.
For $I\in\on{Sub}_n$, let $\on{SYT}(\lambda,I)$ be the set of all standard-Young
tableaux $T$ of shape $\lambda$ with $n+1$ boxes with descent set
$\on{Des}(T) = I$.  Also let $\on{SYT}_I:= \bigcup_{\lambda\in\fP_{n+1}}
\on{SYT}(\lambda,I)$.
According to Theorem~\ref{1.6}, $\kappa_{\lambda,I} = \#\on{SYT}(\lambda,I)$. 

\begin{cor}
Let $I\in\on{Sub}_n$.  The ribbon module $R_I$ is
isomorphic to the direct sum of irreducible submodules of $\sk[S_{n+1}]$:
\[
R_{I} \cong \bigoplus_{T\in \on{SYT}_I}  L_T.
%= \bigoplus_{\lambda\in\fP_{n+1}} \left(\bigoplus_{T\in\on{SYT}(\lambda,I)} L_T\right).
\]
The induced module $M_I$ is isomorphic to the direct sum of irreducible submodules
of $\sk[S_{n+1}]$:
\[
M_{I} \cong \bigoplus_{J\subset I} \left(
\bigoplus_{T\in \on{SYT}_J}  L_T \right).
\]
\end{cor}

%\begin{rem}
%We have 2 isomorphic left submodules of $\sk[S_{n+1}]$:
%the ribbon module $R_I$ and direct sum in the right hand side of the above claim.
%\end{rem}

%\section{Kostka numbers for Coxeter groups}
%\label{coxeter}

\subsection{Kostka numbers for Coxeter groups} 

The symmetric group $S_{n+1}$ is a Coxeter group.
Some of the above constructions can be extended
 to an arbitrary finite Coxeter group $W$.

Let $W$ be a finite Coxeter group of rank $n$ with Coxeter generators $s_1,\dots,s_n$.
Let $\sk[W]$ be the group algebra of $W$ over a field $\sk$ of characteristic 0.

For a subset $I\subset [n]:=\{1,\dots,n\}$, 
let $W_I$ denote the parabolic subgroup of $W$ 
generated by $s_i$, for $i\in I$.
In particular, $W_{[n]}=W$ and $W_\emptyset=\{1\}$.

For $I \subset[n]$, define the {\em induced module} $M_I$ as
the following left $W$-submodule of the group algebra $\sk[W]$
\begin{equation}
\label{eq:induced_W_mod}
M_I := \sk[W] \left(\sum_{w\in\ol{I}} w\right)
\cong\Ind_{W_{\ol{I}}}^W\sk\,,
\end{equation}
where, as usual, $\ol{I} := [n]\setminus I$.

If $W=S_{n+1}$, then the induced module $M_I$ 
is exactly the module $M_\beta \cong 
\Ind_{S_{\beta_1}\times \cdots \times S_{\beta_k}}^{S_{n+1}}\sk$,
where $\beta = \varrho(I)$, from \S\ref{ssec:ind_rib}.

By analogy with~\eqref{eq:Kostka=mult} and 
  Definition~\ref{1.4.2}, we define the $W$-Kostka numbers and the small $W$-Kostka numbers
as follows.

\begin{defn}  
Let $V$ be any irreducible representation of $W$,
and let $I\in \on{Sub}_n$ be any subset of $[n]$.
%{\rm(1)} \ 
The {\em $W$-Kostka number} $K_{V,I}$
is the multiplicity of $V$ in the induced module $M_I$:
\begin{equation}
K_{V,I}:=\dim\Hom_{\sk[W]}(M_I, V).
\end{equation}
%\smallskip
%{\rm(2)} \ 
The {\em small $W$-Kostka number} $\kappa_{V,I}$ is
the alternating sum of $W$-Kostka numbers:
\begin{equation}
\kappa_{V,I}:=\sum_{J\subset I}(-1)^{|J|-|I|}K_{V,J}.
\end{equation}
\end{defn}

Using M\"obius inversion, 
we can express the $W$-Kostka numbers in
terms of the small $W$-Kostka numbers as
\begin{equation}
\label{eq:Mobius_small_Kostka_W}
K_{V,I} = \sum_{J\subset I} \kappa_{V, J}.
\end{equation}

Let $\sk_\epsilon$ be the {\em sign representation} of $W$, which is the 
1-dimensional $\sk[W]$-module, where each Coxeter generator acts as
$s_i\colon x\mapsto -x$.  For a $W$-module $V$, let
\[
V^t : = V\otimes\sk_\epsilon\,.
\]

The following claim generalizes Theorem~\ref{th:small_Kostka_nonneg_symm}.

\begin{thm}
  \label{th:small_W_Kostka_nonneg_symm}
{\rm (1)}
The small $W$-Kostka numbers $\kappa_{V,I}$ are non-negative integers.
\smallskip

{\rm (2)} 
The small Kostka numbers have the symmetry
$\kappa_{V,I} = \kappa_{V^t,\ol{I}}$.
\end{thm}

This theorem follows from Corollary~\ref{cor:M=sum_R_for_W} and
Lemma~\ref{lem:R_Ixsign=RbarI} below.
It would be nice to find a simple
combinatorial rule like Theorem~\ref{1.6} for calculation of the small $W$-Kostka numbers
(e.g. when $W$ is a Weyl group of type $B$ or $D$).

\begin{rem}[R.~Bezrukavnikov]
  As a $W$-module, $R_I$ is a direct sum of the modules given by all the left Kazhdan-Lusztig
  cells with the right descent set $I$. 
\end{rem}

\subsection{Solomon's decomposition of $\sk[W]$}

Given $I\in\on{Sub}_n$, we define the following elements of the group algebra $\sk[W]$
\[
a_I := \sum_{w\in W_I} w 
\qquad 
\textrm{and}
\qquad 
b_I := \sum_{w\in W_I} \epsilon(w)\, w\,,
\]
where $\epsilon(w)=(-1)^{\ell(w)}$ is the sign of $w\in W$
and  $\ell(w):=\min\{l\mid w=s_{i_1}\cdots s_{i_l}\}$ 
is the length of $w$.

Note that the induced module $M_I$ defined by \eqref{eq:induced_W_mod} is
$M_I = \sk[W] \, a_{\ol{I}}$.

Define the {\em ribbon $W$-module} $R_I$ as
\begin{equation}
R_I := \sk[W] \, b_I a_{\ol{I}}.
\end{equation}

If $W=S_{n+1}$, then the ribbon $W$-module $R_I$ 
is exactly the ribbon module~\eqref{eq:ribbon_mod_def}
from \S\ref{ssec:ind_rib}.

\begin{thm} 
\label{th:solomon_dec_W}
\cite[Theorem~2]{so2}
%{\rm (Solomon's decomposition)} \ 
The group algebra $\sk[W]$ decomposes into a direct sum of $2^n$ ribbon 
modules: 
\[
\sk[W] = \bigoplus_{I\in\on{Sub}_n} R_I.
\]
\end{thm}

\subsection{Descent basis}

Let us give a $\sk$-linear basis of $\sk[W]$ which agrees with Solomon's decomposition.

The (right) {\it descent set\/} $\Des(w)$ of an element $w\in W$
is 
\[
\on{Des}(w):=\{i\in I \mid  \ell(w\, s_i) < \ell(w)\}.
\]

For $w\in W$, define the following element of the group algebra

\begin{equation}
d_w := w \, b_{\on{Des}(w)} a_{[n]\setminus\on{Des}(w)}\in\sk[W].
\end{equation}

\begin{example} For $W=S_{3}=\{1,s_1,s_2,s_1s_2, s_2s_1, s_1s_2s_1\}$, we have
\[
\begin{array}{l}
d_{1} = 1 + s_1 + s_2 + s_1 s_2 + s_2 s_1 + s_1 s_2 s_1, \\[0in] 
d_{s_1} =  s_1 (1-s_1)(1 + s_2) 
= - 1 + s_1 -s_2 + s_1 s_2 , \\[0in] 
d_{s_2} = s_2 (1-s_2) (1 + s_1)
= - 1 - s_1 + s_2 + s_2 s_1 ,  \\[0in]
d_{s_1 s_2} = s_1 s_2 (1-s_2)(1+s_1)
= -1 - s_1 + s_1 s_2 + s_1 s_2 s_1, \\[0in]
d_{s_2 s_1} = s_2 s_1 (1-s_1) (1+s_2)
= -1 - s_2 + s_2 s_1 + s_1 s_2 s_1, \\[0in]
d_{s_1 s_2 s_1} 
%= s_1 s_2 s_1 (1 -s_1 - s_2 + s_1 s_2 + s_2 s_1 - s_1 s_2 s_1)
= -1 +s_1 + s_2 - s_1 s_2 - s_2 s_1 + s_1 s_2 s_1.
\end{array}
\]
\end{example}

\begin{thm}
\label{th:des_basis_of_kW}
{\rm (1)}  
The set of elements $\{d_w\mid w\in W\}$ is a $\sk$-linear basis 
of the group algebra $\sk[W]$. 

\smallskip
{\rm (2)}  
For any $I\in\on{Sub}_n$, the set of elements 
$\{d_w\mid w\in W \textrm{ such that } \Des(w) = I\}$ is a $\sk$-linear
basis of the ribbon module $R_I$.
In particular, 
\[
\dim R_I = \#\{w\in W\mid \on{Des}(w) = I\}.
\]

\smallskip
{\rm (3)}  
For any $I\in\on{Sub}_n$, the set of elements 
$\{d_w\mid w\in W \textrm{ such that } \Des(w) \subset I\}$ is a $\sk$-linear
basis of the induced module $M_I$.
In particular,
\[
\dim M_I = \#\{w\in W\mid \on{Des}(w) \subset I\}.
\]
\end{thm}

A proof of this theorem is given in \S\ref{ssec:proof_thm_des_basis_kW} below.
It is immediate from the definitions that $d_w\in R_I$ if $\Des(w)=I$,
and that $d_w\in M_I$ if $\Des(w)\subset I$.

%More generally, we have the following decomposition of induced modules.

\begin{cor}
\label{cor:M_I=sum_of_R_I_for_W}
For $I\in\on{Sub}_n$, we have the decomposition of the induced module $M_I$
into a direct sum of ribbon modules:
\[
M_I = \bigoplus_{J\subset I} R_J.
\]
\end{cor}
%(Here all modules are explicitly embedded into $\sk[W]$.)

\begin{cor} 
\label{cor:M=sum_R_for_W}
For an irreducible $W$-representation $V$ and $I\in\on{Sub}_n$, 
 the small $W$-Kostka number $\kappa_{V,I}$ is the multiplicity of $V$
in the ribbon module $R_I$:
\[
\kappa_{V,I}  = \dim \Hom_{\sk[W]} (V, R_I).
\]
\end{cor}

\begin{proof}
Let $\tilde\kappa_{V,I}$ be the multiplicity of $V$ in $R_I$.
By Corollary~\ref{cor:M_I=sum_of_R_I_for_W}, 
\[
K_{V,I} := \dim \Hom_{\sk[W]}(V,M_I) 
= \sum_{J\subset I} \dim \Hom_{\sk[W]}(V,R_J) 
= \sum_{J\subset I} \tilde\kappa_{V,I},
\]
which is exactly equation~\eqref{eq:Mobius_small_Kostka_W} defining 
the small Kostka numbers. Thus $\tilde \kappa_{V,I} = \kappa_{V,I}$.
\end{proof}

Recall that $\sk_\epsilon$ is the sign representation of $W$.
\begin{lem}
\label{lem:R_Ixsign=RbarI}
We have the isomorphism of $W$-modules:
\[
R_I \otimes \sk_\epsilon \cong R_{\ol{I}}.
\]
\end{lem}

\begin{proof}
By~\cite[Lemma~12]{so2}, we have the isomorphism 
$\sk[W]\, b_J a_I \cong \sk[W] \, a_I b_J$.
Thus $R_I \otimes \sk_\epsilon  = \sk[W]\, a_I b_{\ol{I}} \cong
\sk[W] \, b_{\ol{I}} a_I = R_{\ol{I}}$.
\end{proof}

Now Theorem~\ref{th:small_W_Kostka_nonneg_symm} follows 
from Corollary~\ref{cor:M=sum_R_for_W}
and Lemma~\ref{lem:R_Ixsign=RbarI}.

\subsection{Proof of Theorem~\ref{th:des_basis_of_kW}}
\label{ssec:proof_thm_des_basis_kW}

For $I\subset [n]$, define the (right) 
{\em symmetrization} and the {\em antisymmetrization 
operators} $\Sym_I$ and $\Asym_I$ acting on $\sk[W]$ by 
\[
\Sym_I \colon  f \mapsto f \, a_I
\qquad
\textrm{and}
\qquad
\Asym_I \colon  f \mapsto f \, b_I
\]
Let $\Image(\Sym_I), \Kernel(\Sym_I),
\Image(\Asym_I), \Kernel(\Asym_I) \subset \sk[W]$ be the images
and kernels of these operators.

\begin{lem}
\label{lem:MIJ=quotient}
The ribbon module $R_I$ is isomorphic to the quotient module
\[
R_I \cong 
\Image(\Asym_{I}) / 
(\Kernel(\Sym_{\ol{I}}) \cap \Image(\Asym_{I})).
\]
More precisely, the map $\Sym_{\ol{I}}$ 
restricted to $\Image(\Asym_I)$ 
induces the isomorphism 
between $\Image(\Asym_{{I}}) / 
(\Kernel(\Sym_{\ol{I}}) \cap \Image(\Asym_{I}))$
and $R_I$.
\end{lem}

\begin{proof}
We have $R_I := \sk[W] \, b_I \, a_{\ol{I}} 
= \{ f \, a_{\ol{I}} \mid f \in \Image(\Asym_I)\}
=\{\Sym_{\ol{I}}(f)\mid f\in \Image(\Asym_{I})\}$.
\end{proof}

For $i\in [n]$, define the subspaces 
$S^i :=  \{f\in \sk[W] \mid f s_i = f \}$
and $A^i :=  \{f\in \sk[W] \mid f s_i = - f \}$ of the group algebra $\sk[W]$.
Clearly, $S^i = \Image(\Sym_{\{i\}}) = \Kernel(\Asym_{\{i\}})$ and 
$A^i = \Image(\Asym_{\{i\}}) = \Kernel(\Sym_{\{i\}})$. 

\begin{lem}
\label{lem:image_kernel}
The image of the symmetrization operator $\Sym_I$
is the intersection of subspaces:
\[
\Image(\Sym_I) = \bigcap_{i\in I} S^i,
\]
and the kernel of $\Sym_I$ is the linear span of subspaces:
\[
\Kernel(\Sym_I) = 
\Span{i\in I} A^i.
\]
Similarly, the image and the kernel of the antisymmetrization 
operator  $\Asym_I$ are
\[
\Image(\Asym_I) = \bigcap_{i\in I} A^i
\qquad
\textrm{and}
\qquad
\Kernel(\Asym_I) = 
 \Span{i\in I} S^i.
\]
\end{lem}

\begin{proof}
An element $f\in\sk[W]$ belongs to the image of $\Sym_I$ 
if and only if $f \, w = f$, for any $w\in W_I$.
Equivalently, we should have $f\, s_i = f$, for all $i\in I$,
that is, $f\in \bigcap_{i\in I} S^i$.

An element $g = \sum_{w\in W} g_w\, w \in \sk[W]$ belongs 
to the kernel of $\Sym_I$ if and only if
the sum of its coefficients $g_w$ over any 
right $W_I$-coset $C$ in $W$ is zero: $\sum_{w\in C} g_w = 0$.
Such $g$ should be a linear combination of the elements 
$w - w s_i$, for $w\in W$ and $i\in I$.
This follows from the fact that the right $W_I$-cosets in $W$
are exactly the connected components of the graph on set of
vertices $W$ with the set of edges 
$\{(w,w s_i) \mid w\in W, \, i\in I\}$.
For fixed $i\in I$, the elements $w-w\,s_i$, for $w\in W$, 
span the subspace $A^i$.  Thus $g$ should belong to the linear 
span of the subspaces $A^i$ over all $i\in I$.

The claim about the image and the kernel of $\Asym_I$ is proved
analogously.
\end{proof}

For $w\in W$, let 
\[
c_w := w\, b_{\Des(w)}\in\sk[W]. 
% = w\, \left(\sum_{v\in W_{\Des(w)}} v \right).
\]

\begin{lem}  The set $\{c_w \mid w\in W\}$ is a $\sk$-linear basis of $\sk[W]$.
\end{lem}

\begin{proof}
Notice that $w$ is the unique maximal by length element in the 
expansion of $c_w$.
Thus the set of elements $\{c_w| w\in W\}$ is related to the standard linear
basis $\{w \mid w\in W\}$ of $\sk[W]$ by a triangular matrix with 1's
on the diagonal.  Thus $\{c_w\mid w\in W\}$ is a linear basis of $\sk[W]$.
\end{proof}

\begin{lem}
\label{lem:ImKerInNewBasis}
In the linear basis $\{c_w\mid w\in W\}$, 
the subspace $A^i$ is the coordinate subspace of $\sk[W]$ 
spanned by the basis elements $c_w$, for 
all $w$ such that $i\in \Des(w)$:
\[
A^i = \Span{w\, \mid\, i\in \Des(w)} c_w.
\]
Moreover, for any $I\subset [n]$, 
$\Image(\Asym_I)$ and $\Kernel(\Sym_I)$ are the coordinate subspaces given by
\[
\Image(\Asym_I)  = \Span{w \,\mid\, I \subset \Des(w)} c_w
\qquad
\textrm{ and }
\qquad
\Kernel(\Sym_I)  = \Span{w \,\mid\, I \cap \Des(w)
\ne\emptyset} c_w.
\]
\end{lem}

\begin{proof}
Clearly, $c_w\in A^i$ if $i\in \Des(w)$.
Since $\#\{w\in W\mid i\in \Des(w)\} = |W|/2 = \dim A^i$,
we deduce that the basis elements $c_w$ such that $i\in \Des(w)$ 
span the subspace $A^i$.
Now the claims about $\Image(\Asym_I)$ and $\Kernel(\Sym_I)$
follow from Lemma~\ref{lem:image_kernel}.
\end{proof}

%\begin{proof}[Proof of Proposition~\ref{prop:dim_MIJ}]

Now we can prove part (2) of Theorem~\ref{th:des_basis_of_kW}.
By Lemmas~\ref{lem:MIJ=quotient} and~\ref{lem:ImKerInNewBasis}, 
the map $\Sym_{\ol{I}}$, 
restricted to $\on{Span}\left\{c_w \mid I\subset \Des(w)\right\}$,
induces the isomorphism 
\[
\left(
\Span{w\,\mid\, I\subset \Des(w)} c_w
\right) 
/
\left(
\Span{w\,\mid\, I\subset \Des(w),\ \ol{I} \cap \Des(w)\ne \emptyset} c_w
\right)
\iso
R_I . 
\]
%
%induces the isomorphism between this quotient space and $R_I$.
Thus the elements $\Sym_{\ol{I}} (c_w)$, 
for all $w\in W$ such that $\Des(w) = I$, form a linear basis of $R_I$.  
This is exactly the claim of Theorem~\ref{th:des_basis_of_kW}(2),
because $\Sym_{\ol{I}} (c_w) = w \, b_{\Des(w)} a_{\ol{I}} =: d_w$. 

%w \, b_{\Des(w)} a_{[n]\setminus \Des(w)} 

Part (1) of Theorem~\ref{th:des_basis_of_kW} follows from part (2) and 
Solomon's decomposition (Theorem~\ref{th:solomon_dec_W}).

To prove part (3) Theorem~\ref{th:des_basis_of_kW}, notice that
%, by the definitions,
$d_w\in M_I$, for all $w\in W$ such that $\Des(w)\subset I$. 
Such  $w$'s are exactly the minimal length coset representatives for the
right cosets $W/W_{\ol{I}}$.
So we get $|W|/|W_{\ol{I}}|$ linearly independent elements $d_w$ in the space 
$M_I\cong \Ind_{W_{\ol{I}}}^W \sk$ of dimension $|W|/|W_{\ol{I}}|$.
Thus they form a $\sk$-linear basis of $M_I$.

\section{Kostka sheaves}
\label{3}
Let $\CC=\Rep(S_{n+1}) = \sk[S_{n+1}]\modu$ denote  the category  of finite dimensional representations of 
$S_{n+1}$ over $\sk$, that is the category of finite dimensional left $\sk[S_{n+1}]$-modules.
The key object of this section is a semisimple amonodromic perverse sheaf $\CF$ smooth along the
coordinate stratification of $\BC^n$. Namely, it is a direct sum of constant sheaves $\CF_I$ on
the strata closures $\ol{S}_I,\ I\subset\on{Sub}_n$. Finally,
\[\CF_I=\bigoplus_{\lambda\in\fP_{n+1}} V_\lambda^{\oplus\kappa_{\lambda,I}}[n-|I|].\]

\subsection{The sheaf $\CF$ in GGM realization}
The corresponding amonodromic semisimple GGM sheaf with values in $\CC$ is 
\[\scP = \scP^{(n)}\in\GGM^\ass_n(\Rep(S_{n+1})).\] 

By definition, for any $I\in\on{Sub}_n$ we have $\scP(I) = R_I$,
the ribbon module defined by
%\eqref{(1.7.1)}.  
\eqref{eq:ribbon_mod_def} and \eqref{eq:ind_rib_I}.

\subsection{The sheaf $\CF$ in hyperbolic realization}
The corresponding amonodromic semisimple hyperbolic sheaf with values in $\CC$ is
\[\scQ =\scQ^{(n)}\in\Hyp^\ass_n(\Rep(S_{n+1})).\]
Explicitly, by~\S\ref{2.8} its hyperbolic stalks are
\[\scQ(C_{(I,\epsilon)}) = \oplus_{J\subset I} R_J.\]
Therefore, 
by Theorem~\ref{th:ind_decom_ribbon},
%by~Corollary~\ref{1.7.1}, 
they are isomorphic to the induced modules 
\begin{equation}
  \label{3.2.1}
\scQ(C_{(I,\epsilon)})=\scQ(I)\simeq M_I = M_{\varrho(I)}
\end{equation}
We recall that $\varrho$ denotes the isomorphism
$\varrho\colon\on{Sub}_n\iso\fQ_{n+1}$ of~Lemma~\ref{1.1.1},
so that $\varrho(I) = \alpha$ is a composition of $n + 1$, and 
$M_\alpha$ is the $S_{n+1}$-module induced from the trivial representation 
of the subgroup $S_\alpha \subset S_{n+1}$.

\subsection{A functor $\Rep(S_{n+1})\to\Hyp^\ass_n(\Rep(S_{n+1}))$}
\label{3.2.2}
More generally, for any $M\in\Rep(S_{n+1})$ we define a hyperbolic sheaf 
\[\scQ_M = (M\otimes\scQ(C_{(I,\epsilon)}), \Id_M\otimes \gamma, \Id_M\otimes \delta).\]
This way we get a ``localization''  functor
\[\scQ\colon\Rep(S_{n+1}) \to\Hyp^\ass_n(\Rep(S_{n+1})).\]

\begin{thm}
  \label{key}
For $M\in\Rep(S_{n+1})$, we have an isomorphism $\on{FT}\scQ_M\simeq\scQ_{M\otimes\sk_\epsilon}$.  
\end{thm}

\begin{proof} 
It suffices to construct an isomorphism $\on{FT}\scQ\simeq\scQ\otimes\sk_\epsilon$.
Equivalently, in the GGM realization, we have to construct an isomorphism
$\on{FT}\scP\simeq\scP\otimes\sk_\epsilon$. The existence of the desired isomorphism follows
immediately from the isomorphism
$V_\lambda\otimes\sk_\epsilon\simeq V_{\lambda^t}$,
%Lemma~\ref{symmetry}
Theorem~\ref{th:small_Kostka_nonneg_symm}, and
the formula for Fourier-Sato transform in~\S\ref{2.2.1}.
%HERE WE IDENTIFY THE COMPLEXES OF VANISHING CYCLES~\eqref{(2.6.1)} FOR THE COUPLE 
%$(\CP, \CQ)$ WITH THE KNOWN PARABOLIC COMPLEXES, SEE~Theorem~\ref{thm 3.6} BELOW
\end{proof}

\subsection{Induced from parabolics as functions on flags}
\label{3.3}
Let us call a {\em type}  
a sequence of integers $\chi = (\chi_1 , \ldots , \chi_p)$ with
\[1\leq \chi_1 < \ldots < \chi_p \leq n + 1.\]
Notation: $\{\chi\} = \{\chi_1 , \ldots , \chi_p\}$. 

The number $p =\ell(\chi)$ will be called the length of $\chi$. The set of types of length $p$ 
will be denoted $\CT yp_p$; this set contains $\binom{n+1}{p}$ elements.  

To each type corresponds a composition of $n+1$
\begin{equation}
  \label{(3.2.1)}
  \alpha(\chi) = (\chi_1, \chi_2 - \chi_1, \ldots, n + 1 - \chi_p)\in \fQ_{n+1}
\end{equation}

A {\em flag of type $\chi$} is a chain $I_\bullet$ of subsets 
$I_1\subset \ldots \subset I_p\subset [n + 1]$ with $|I_i| = \chi_i$, and $p$ is the length of $I_\bullet$.

We denote by $F\ell_\chi$ the set of all flags of type $\chi$,
and by $F\ell_p$ the set of all flags of length $p$.

The set $F\ell_\chi$ is acted upon from the left by $S_{n+1}$ and is isomorphic to 
$F\ell_\chi \cong S_{n+1}/S_\chi,$
where a parabolic subgroup $S_\chi$ is the stabilizer of the standard flag 
\[[\chi_1] \subset  \ldots  \subset [\chi_p].\]
We have
\[S_\chi \cong S_{\alpha_1}\times S_{\alpha_2} \times \ldots \times S_{\alpha_{p+1}} = S_\alpha,\]
where
$\alpha = \alpha(\chi)$, cf.~\eqref{(3.2.1)}.

Let $M_\chi$ denote the space of maps of sets
\[\on{Maps}_{\on{Sets}}(F\ell_\chi, \sk) = \{ f\colon F\ell_\chi \to \sk\};\]
it is an $S_{n+1}$-representation isomorphic to $M_\alpha$ introduced~\S\ref{ssec:ind_rib}.

This realization of the modules $M_\alpha$ is convenient for describing some natural 
morphisms between them. 

For instance, there are $(n+1)!$ flags of length $n+1$, all of them having type
$\chi_0 = (1,2,\ldots, n+1)$. The corresponding representation $M_{\chi_0}$ is the regular one.  
We postulate that there is a single type $\emptyset$ of length zero, and define 
$M_\emptyset$ to be the trivial representation $\sk$.  

\subsection{Induction and restriction}
\label{3.4}
For two types $\chi, \theta$ we write $\theta\subset \chi$ if $\{\theta\}\subset \{\chi\}$. 

We have obvious maps $\partial_{\chi\theta}\colon F\ell_\chi \to F\ell_\theta,$
wherefrom we obtain the ``restriction'', or pullback, morphisms in $\Rep(S_{n+1})$
\[r_{\theta\chi} = \partial_{\chi\theta}^*\colon M_\theta \to M_\chi.\]
Their adjoint are ``induction'', or pushout, morphisms 
\[i_{\chi\theta} = \partial_{\chi\theta *}\colon M_\chi \to M_\theta\]
are explicitly defined as follows: for $f\in M_\chi,\ I_\bullet\in F\ell_\theta$,  
\begin{equation}
  \label{3.4.1}
  i_{\chi\theta}(f)(I_\bullet) = \sum_{J_\bullet\in \partial_{\chi\theta}^{-1}(I_\bullet)} f(J_\bullet)
\end{equation} 

\subsection{Parabolic complexes}
\label{3.5}
These are two dual complexes of length $n$ in $\Rep(S_{n+1})$.

The {\em master complex} going back at least to~\cite[(1.2)]{ka} (specialized
to the case when the Weyl group $W=S_{n+1}$, the module $M$ is trivial, and $q=1$) is
\begin{equation}
  \label{3.5.1}
  \CM as^\bullet\colon 
0 \to M_{\chi_0} \to   \oplus_{\chi: |\chi| = n - 1} M_\chi \to \ldots 
\to \oplus_{\chi: |\chi| = 1} M_\chi
\to M_\emptyset \to 0
\end{equation}
The differentials 
\[d_p\colon \oplus_{\chi: |\chi| = p} M_\chi \to \oplus_{\chi: |\chi| = p - 1} M_\chi\]
are induced by the maps $i_{\chi\theta}$ with appropriate signs. More precisely, 
their nonzero matrix elements are the maps $d_{\chi\chi'}\colon M_\chi \to M_{\chi'}$
with $\chi'\subset \chi,\ |\chi'| = |\chi| - 1$. Given a type $\chi$ of length $p$, there are $p$
subtypes $\chi' = \partial_i\chi,\ 1\leq i\leq p$, of length $p-1$. 

By definition, $d_{\chi,\partial_i\chi} = (-1)^i i_{\chi,\partial_i\chi}.$
We consider the ``master complex''~\eqref{3.5.1} as concentrated in degrees $[0,n]$. 

The {\em adjoint master complex} is a similar one with differentials induced by
the restriction maps $r_{\chi', \chi},\ |\chi'| =  |\chi| - 1$ (with signs):
\begin{equation}
  \label{3.5.2}
\CM as^{\prime\bullet}\colon  
0 \to M_\emptyset \to \oplus_{\chi: |\chi| = 1} M_\chi \to \ldots \to M_{\chi_0} \to 0 
\end{equation}

By~\cite[($\#$) at page 942]{ka} applied to the case when the Weyl group $W=S_{n+1}$,
the module $M$ is trivial, and $q=1$,
\begin{equation}
  \label{3.5.3}
H^i(\CM as^{\bullet}) = 0,\ i > 0,
\end{equation}
and the only nonzero cohomology is 
\begin{equation}
  \label{3.5.4}
H^0(\CM as^{\bullet}) = \sk_\epsilon, 
\end{equation}
the sign representation.

\begin{cor}
  \label{3.5.5}
  $H^i(\CM as^{\prime\bullet}) = 0,\ i > 0.$
\end{cor}

\subsection{Comparison with the vanishing cycles master complex}
\label{3.6}
Recall the complex $\Phi^\bullet(\scQ)$ introduced in~\eqref{(2.6.1)} for arbitrary hyperbolic
sheaf $(E,\gamma,\delta)$.

\begin{thm}
  \label{thm 3.6}
  There exists an isomorphism of complexes in $\Rep(S_{n+1})$
  \begin{equation}
    \label{(3.6.1)}
    \CM as^\bullet \iso \Phi^\bullet(\scQ). 
  \end{equation}
\end{thm}

\begin{proof} 
We know already that the individual terms of the above complexes are isomorphic, 
see~\eqref{3.2.1}. Both of them start with the regular representation 
$M_{[n]} = M_{(11\ldots 1)} = \sk[S_{n+1}]$. 
By~\eqref{3.5.4}, $H^0(\CM as^{\bullet}) = \sk_\epsilon$.

On the other hand, $H^0(\Phi^\bullet(\scQ)) \cong \sk_\epsilon$. Indeed, the Euler characteristic of
$\Phi^\bullet(\scQ)$ is $\sk_\epsilon$ e.g.\ by~\cite[Theorem 2]{so}, and
$H^{>0}(\Phi^\bullet(\scQ))=0$ by~Proposition~\ref{thm 2.6.1}.

We will apply the following elementary 
\begin{lem}
  \label{3.6.2}
  Let $\CC$ be a semisimple category, $A\subset B, A'\subset B'$ 
objects of $\CC$. If $A \cong A'$ and $B \cong B'$ then $A/B \cong A'/B'$. \hfill $\Box$
\end{lem}

We can reformulate the lemma as follows. Choose an isomorphism
$\phi_A\colon A\iso A'$. Since $\CC$ is semisimple, there are objects $C, C'$ such that
$B \cong A\oplus C$ and $B' \cong A'\oplus C'$. According to the lemma, $C\cong C'$. 
Whence there exists an isomorphism $\phi_B\colon B\iso B'$ making the square
\[\begin{CD} A @>>> B \\
@V{\phi_A}VV @V{\phi_B}VV \\
A' @>>> B'
\end{CD}\]
commutative. 

Now we can construct the desired isomorphism $\CM as^\bullet \iso \Phi^\bullet(\scQ)$ inductively from 
left to right, using the isomorphism $H^0(\Phi^\bullet(\scQ^{(n)})) \cong \sk_\epsilon$
and the acyclicity of both complexes in positive degrees.
\end{proof}

\subsection{Relation to the permutohedron}
\label{3.6.3}
We present a geometric interpretation of the master complex of vanishing cycles
$\Phi^\bullet(\scQ)$, cf.~\cite{ka,so,wi}.

\begin{prop}
  \label{prop 3.6.3}
  The complex $\Phi^\bullet(\scQ)$ is isomorphic, as a complex of 
  $\sk$-vector spaces, to the complex $C^\bullet(\on{Perm}_n)$ of cellular cochains\footnote{of
  the polyhedral cellular structure of $\on{Perm}_n$} of the $n$-th permutohedron.
\end{prop}

For instance, let $n = 2$. The complex $\Phi^\bullet(\scQ)=\Phi^\bullet(\scQ^{(2)})$ has the form
\[0 \to \sk[S_3] \to \sk[S']\oplus \sk[S''] \to \sk \to 0.\]
Here $S', S''\cong S_2$ are the subgroups of $S_3$ generated by transpositions $s_1 = (12)$
and $s_2 = (23)$ respectively. 
So the dimensions of its terms  are~$6,6,1$, and it is isomorphic to the complex of cochains of the 
hexagon $\on{Perm}_2$. 

\begin{proof}
Our complex $\Phi^\bullet(\scQ)$ has the form
\[0 \to \sk[W] \to \oplus_{i\in S} \sk[W/W_{\{ i\}}] \to  \oplus_{\{ i,j\}\subset S}\sk[W/W_{\{ i,j\}}] \to
\ldots \to \sk \to 0,\]
where $W=S_{n+1}$, and for a subset $I\subset S$ of the set of Coxeter generators, $W_I \subset W$
denotes the subgroup generated by $s_i,\ i\in I$. 

On the other hand, consider the root arrangement in $\BR^S$; the permutohedron 
$P_W$ is a polygon dual to this arrangement. Thus, its vertices are in one-to-one correspondence
with the chambers that form a set on which $W$ acts simply transitively, i.e. after choosing a
fundamental chamber $C$ we can identify this set with $W$, i.e. with a base of $\sk[W]$. 

Next, the edges of  $P_W$ are in bijection with walls; there are $n$ walls $M_i$ on the boundary of
$C$ which   correspond to simple reflections $s_i$; the stabilizer of $M_i$ 
being the subgroup $W_{\{ i\}}$. So the set of walls is in bijection with 
$\coprod_{i\in S} W/W_{\{ i\}}$, etc.
\end{proof}

%\red{``INVERSE THEOREM'': FROM PARABOLIC COMPLEXES TO HYPERBOLIC SHEAVES}

\subsection{Warning}
\label{3.7}
One might think that the maps $i_{\chi\theta}$ and $r_{\theta\chi}$ give rise to a hyperbolic sheaf,  
but this is not so: the collinearity axiom does not hold, as the example of $S_3$ already shows.

Namely, let $n = 2$. Consider the square
\[\begin{CD}M_{(1)} @>{r}>> M_\emptyset \\
@V{i}VV @V{i}VV \\
M_{(12)} @>{r}>>  M_{(2)}.
\end{CD}\]
Let $f\in M_{(1)}$, i.e.\ $f$ is a function from one-element subsets of $[3]$ to $\sk$. 
Then $r(f) = \sum_{i=1}^3 f(i)$, and $ir(f)$ is the  constant function with value $r(f)$.

On the other hand, $i(f)(\{i\}\subset \{i, j\} ) = f(i)$,
and $ir(f)(\{i, j\} ) = f(i) + f(j)$.
Thus, $ir \neq ri$, the square is not commutative.

\section{From $S_{n+1}$ to $\GL(n+1,\BF_q)$}
We denote $G = G_{n+1} =\GL_{n+1}(\BF_q)$.

\subsection{Flag spaces and induced representations}
\label{4.1}
We fix a vector space $U = \BF_q^{n+1}$. 
For a type $\chi = (\chi_1 , \ldots , \chi_p)\in \CT yp_p$ (notation of~\S\ref{3.3}),
a flag of type $\chi$ is a chain of linear subspaces
\[F_\bullet\colon 0 = U_0 \subset U_1 \subset \ldots \subset U_p \subset U;\]
such that $\dim U_i=\chi_i$. We will call $p$ the length of $F_\bullet$.

We denote by $F\ell_\chi(q)$ the set of all flags of type $\chi$, and by $F\ell_p(q)$
the set of all flags of length $p$.

The set $F\ell_\chi(q)$ is acted upon from the left by 
$G$ and is a homogeneous set $F\ell_\chi(q) \cong G/P_\chi$,
where $P_\chi$ is the standard parabolic subgroup,  the stabilizer of the standard flag 
$\BF_q^{\chi_1} \subset  \ldots  \subset \BF_q^{\chi_p}.$ 

For $\chi = (1,2,\ldots,n,n+1)$, we have $P_\chi = B$, the standard Borel.  

Let $M_\chi(q)$ denote the space of maps of sets
\[\on{Maps}_{\on{Sets}}(F\ell_\chi(q),\sk) = \{ f\colon F\ell_\chi(q)\to\sk\};\]
it is a $G$-module isomorphic to $\Ind_{P_\chi}^G\sk$.

\subsection{Induction and restriction}
\label{4.2}
We have the natural $q$-analogues of the master complex and its adjoint of~\S\ref{3.5}.

Let  $\theta\subset \chi$ be two types.  
We have obvious maps 
\[\partial_{\chi\theta}\colon F\ell_\chi(q)\to F\ell_\theta(q),\]
wherefrom the restriction, or pullback,  morphisms in $\Rep(G)$
\[r_{\theta\chi} = \partial_{\chi\theta}^*\colon M_\theta(q)\to M_\chi(q).\]
Their adjoint are induction, or pushout,  morphisms 
\[i_{\chi\theta} =\partial_{\chi\theta *}\colon M_\chi(q)\to M_\theta(q),\]
that are explicitly defined as follows: for $f\in M_\chi(q), F_\bullet\in F\ell_\theta$,  
\begin{equation}
  \label{4.2.1}
  i_{\chi\theta}(f)(F_\bullet) = \sum_{F'_\bullet\in\partial_{\chi\theta}^{-1}(F_\bullet)} f(F'_\bullet).
\end{equation}

Proceeding as in~\S\ref{3.5} we obtain a parabolic complex
\begin{equation}
  \label{4.3.1}
\on{DL}(\sk)^\bullet: 
0 \to M_{\chi_0}(q) \to   \oplus_{\chi: |\chi| = n - 1} M_\chi(q) \to \ldots 
\to  M_\emptyset(q) \to 0
\end{equation}
This is nothing but the Deligne-Lusztig complex~\cite[(1.2)]{dl} for the case of trivial
representation $E=\sk$ of $\GL(n+1,\BF_q)$.
According to~\cite[Theorem in~\S2]{dl},
\begin{equation}
  \label{4.3.2}
H^i(\on{DL}(\sk)^{\bullet}) = 0, \ i > 0.
\end{equation}
The only nonzero cohomology is 
$H^0(\on{DL}(\sk)^\bullet) =\on{St}$, the Steinberg module~\cite{st}.

\subsubsection{Example $n=1$} 
Consider the following hyperbolic sheaf $E \in \CM(\BC, 0)$ with values 
in $\Rep(\GL(2,\BF_q))$:
\[E_0 = \Ind^G_B\sk =\on{Maps}(\BP^1,\sk),\ E_+ = E_- =\sk.\]
We have
\[\Ind^G_B\sk \cong\sk \oplus \on{St}.\]

The Fourier-Sato transform $\on{FT}(E)$ has the hyperbolic stalks
\[E_0^\vee = E_0 = \Ind^G_B\sk, \ 
E_+^\vee = E_-^\vee = \on{St}.\]

\subsubsection{Example $n=2$}
We have $4 = 2^2$ standard parabolics $G, P_1, P_2, B$ with
\[G/G = \{*\}, G/P_1 = \BP^2, G/P_2 = \BP^{2\vee}.\]
The real stratification $\CS_\BR$ of $\BR^2$ has~4 2-dimensional cells $C_i^\pm,\ i=1,2$;
and~4 1-dimensional cells $\ell_i^\pm,\ i=1,2$, and the 0-dimensional cell~$\{0\}$.

We define a hyperbolic sheaf $E \in \CM(\BC^2, \CS)$ with hyperbolic stalks
\[E_{C_i^\pm} =\sk,\]
\[E_{\ell_1^\pm} = \Ind^G_{P_1}\sk =\on{Map}(\BP^2,\sk),\ 
E_{\ell_2^\pm} = \Ind^G_{P_2}\sk =\on{Map}(\BP^{2\vee},\sk),\]  
\[E_0 = \Ind^G_B\sk =\on{Map}(G/B,\sk).\] 
We have the decomposition into irreducibles
\[E_0 \cong\sk \oplus \on{St}'\otimes\sk^2 \oplus \on{St}.\]
We have $\dim\on{St}' = q^2 + q,\ \dim\on{St}=q^3$, and $|G/B|=q^3+2q^2+2q+1$. 

The Fourier-Sato transform $\on{FT}(E)$ has hyperbolic stalks
\[E_0^\vee = E_{C_i^\pm},\ E_{\ell_1^\pm}^\vee = E_{\ell_2^\pm},\ E_{\ell_2^\pm}^\vee = E_{\ell_1^\pm},\
E_{C_i^\pm}^\vee = E_0.\]

\subsection{Unipotent representations of $\GL(n+1,\BF_q)$}
Let $\CH_{n+1}$ denote the Iwahori-Hecke algebra of $S_{n+1}$ at $q$. There is an
isomorphism~\cite{bc} of algebras $\sk[S_{n+1}]\simeq\CH_{n+1}$. At the price of
extending $\sk$ by $\sqrt{q}$, one can choose an explicit isomorphism of~\cite{l}.
It gives rise to an equivalence of semisimple abelian categories
$\CC=\Rep(S_{n+1})\cong\Rep(\CH_{n+1})=:\CC_q$. Under this equivalence, the master complex
$\CM as^\bullet$ goes to the complex $_q\CM as^\bullet$ of~\cite[(1.2)]{ka} (specialized to the
case when the Weyl group $W=S_{n+1}$, and the representation $M$ is trivial).

Furthermore, $\CC_q$ is canonically equivalent to the semisimple abelian category
$\CC_G:=\Rep_{\on{unip}}(G_{n+1})$ of unipotent representations of $\GL_{n+1}(\BF_q)$: direct sums of
constituents of the natural representation of $\GL_{n+1}(\BF_q)$ in the space of functions on the
flag variety. Under this equivalence, $_q\CM as^\bullet$ goes to the Deligne-Lusztig complex~\cite{dl}
$\on{DL}(\sk)$ of the trivial $G$-module $\sk$.

More generally, composing the functor $\scQ$ of~\S\ref{3.2.2} with the above equivalence
$\CC_G\simeq\CC$, we obtain a functor
\[\scQ^G\colon\Rep_{\on{unip}}(G_{n+1})\to\Hyp^\ass_n(\Rep_{\on{unip}}(G_{n+1})).\]
Now given a unipotent representation $K\in\Rep_{\on{unip}}(G_{n+1})$, we consider the complex
of vanishing cycles $\Phi^\bullet(\scQ^G(K))$ of~\eqref{(2.6.1)}.
Then it follows from~Theorem~\ref{thm 3.6} that $\Phi^\bullet(\scQ^G(K))$ is isomorphic to
the Deligne-Lusztig complex $\on{DL}(K)$.
Finally, it follows from~Theorem~\ref{key} that we have an isomorphism
$\on{FT}\scQ^G(K)\simeq\scQ^G(K^\vee)$, where $K^\vee$ stands for the Deligne-Lusztig dual of $K$.

\end{document}